\documentclass[12pt]{amsart}
\usepackage{amssymb, eucal, amsfonts, amsmath, xypic,latexsym}
\usepackage{amsfonts}
\usepackage{mathrsfs}

\input {cyracc.def}

\newtheorem{lem}{Lemma}
\newtheorem{thm}{Theorem}
\newtheorem{prop}{Proposition}
\newtheorem{rem}{Remark}
\newtheorem{cor}{Corollary}

\newcommand{\C}{\mathbb{C}}
\newcommand{\Z}{\mathbb{Z}}
\newcommand{\N}{\mathbb{N}}

\newcommand{\K}{\mathbb{K}}
\newcommand{\LL}{\mathfrak{q}}

\newcommand{\tcx}{\widetilde{\mathcal{I}}}

\newcommand{\g}{\mathfrak{g}}
\newcommand{\hh}{\mathfrak{h}}

\newcommand{\n}{\mathfrak{n}}

\newcommand{\gl}{\mathfrak{gl}}
\newcommand{\h}{\mathfrak{k}}

\newcommand{\pp}{\mathfrak{p}}
\newcommand{\tg}{\mathfrak{q}}

\newcommand{\rank}{\text{rank}}
\newcommand{\Specm}{\text{Specm}}
\newcommand{\ZZ}{\mathcal{Z}}
\newcommand{\ind}{\textnormal{ind}}

\newcommand{\ad}{\text{ad}}

\newcommand{\im}{\text{Im}}
\newcommand{\Hom}{\text{Hom}}

\newcommand{\Ad}{\textnormal{Ad}}
\newcommand{\gr}{\text{gr }}
\newcommand{\spn}{\text{span}}
\newcommand{\Lie}{\text{Lie}}

\newcommand{\codim}{\text{codim}}

\newcommand{\sgn}{\text{sgn}}
\newcommand{\chr}{\text{char}}
\newcommand{\Ind}{\text{Ind}}

\renewcommand{\th}{\text{th}}

\newcommand{\Zas}{\mathcal{Z}}
\newcommand{\X}{\mathcal{I}}
\newcommand{\CC}{\mathcal{C}}
\newcommand{\m}{\mathfrak{m}}

\newcommand{\Sf}{\mathfrak{S}}

\newcommand{\II}{\mathcal{D}}
\newcommand{\tG}{Q}
\newcommand{\gS}{\mathfrak{S}}
\newcommand{\BB}{\mathcal{B}}
\newcommand{\V}{\mathcal{W}}

\newcommand{\ra}{\rightarrow}
\newcommand{\tR}{R}
\newcommand{\tr}{\mathfrak{r}}

\title{Invariants of Centralisers in Positive Characteristic}
\author{Lewis W. Topley}
\begin{document}
\maketitle

\begin{abstract}
Let $\tG$ be a simple algebraic group of type $A$ or $C$ over a field of good positive characteristic. We show for any $x \in \tg =\Lie(Q)$ that the invariant algebra $S(\tg_x)^{\tg_x}$ is generated by the $p^\th$ power subalgebra and the mod $p$ reduction of the characteristic zero invariant algebra. The latter algebra is known to be polynomial \cite{PPY} and we show that it remains so after reduction. Using a theory of symmetrisation in positive characteristic we prove the analogue of this result in the enveloping algebra, where the $p$-centre plays the role of the $p^\th$ power subalgebra. In Zassenhaus' foundational work \cite{Zas}, the invariant theory and representation theory of modular Lie algebras were shown to be explicitly intertwined. We exploit his theory to give a precise upper bound for the dimensions of simple $\tg_x$-modules.

When $\g$ is of type $A$ and $\g = \h \oplus \pp$ is a symmetric decomposition of orthogonal type we use similar methods to show that for every nilpotent $e \in \h$ the invariant algebra $S(\pp_e)^{\h_e}$ is generated by the $p^\th$ power subalgebra and $S(\pp_e)^{K_e}$ which is also shown to be polynomial.
\end{abstract}

\section{Introduction}
In \cite{PPY} Premet forumlated the conjecture: if $\tG$ is a reductive group over $\C$ and $x\in \tg$ then $S(\tg_x)^{\tG_x}$ is a polynomial algebra on $\rank(\tg)$ variables. The authors of \cite{PPY} went on to confirm the conjecture for simple groups of type $A$ and $C$. Some partial results were obtained in types $B$ and $D$. A counterexample was found by Yakimova in \cite{Yak3}; here $x$ is a long root vector in type $E_8$. It is hopeful that the conjecture will hold in type $B$ and $D$ although the currently known methods are insufficient to present a full set of invariants in all cases. In a completely separate work \cite{BB} Brown and Brundan proved the conjecture again in type $A$ using very different methods. We shall go into greater depth and compare the two approaches below.

The purpose of this article is to bring the above invariant theoretic discussion into the characteristic $p$ realm, and exploit some combinatorial techniques to study the representation theory of the centralisers $\tg_x$ in type $A$ and $C$. When an algebraic group may be reduced modulo $p$, in an appropriate sense, we have groups $\tG$, $\tG_p$ and their respective Lie algebras $\tg$, $\tg_p$. If $\tG$ is reductive and the characteristic of the field is very good for $\tG$ then it is known that $S(\tg_p)^{\tg_p}$ is generated by $S(\tg_p)^p$ and the mod $p$ reduction of $S(\tg)^{\tG}$, and that similar theorem holds for the invariants in the enveloping algebra. This is the most concise description of the algebra of invariants which we could hope for, and Kac conjectured that it should hold in general \cite{Kac}. A simple counterexample is given by the three dimensional solvable algebraic Lie algebra spanned by $\{h, a, b\}$ with non-zero brackets $[h,a] = a$ and $[h,b] = b$, where $a^{p-1} b$ is a example of an invariant which is not generated in this way. Despite failing in general, we shall show that this nice behaviour holds for centralisers in type $A$ and $C$, thus giving us a complete description of the symmetric invariant algebras in these cases.

Let us now introduce the notation required to state our first theorem. These notations shall be used from henceforth without exception. Let $N \in \N$, let $\K$ be an algebraically closed field of characteristic $p > 0$ and let $V$ be an $N$-dimensional vector space. The group $G = GL(V)$ acts by conjugation on its Lie algebra $\g = \mathfrak{gl}(V)$. Choose a non-degenerate bilinear form $(\cdot, \cdot) : V \oplus V \ra \K$ which is either symmetric or skew. We write $(u, v) = \epsilon (v, u)$ with $\epsilon = \pm 1$. The subgroup of $G$ preserving the form shall be denoted $K$, and is either an orthogonal group or a symplectic group. Whenever we discuss such a group, we shall assume that $\chr(\K) \neq 2$. The Lie algebra shall be denoted $\h$ and is equal to the set of all $x \in \g$ which are skew self-adjoint with respect to $(\cdot,\cdot)$. If we choose a basis for $V$ then $(\cdot, \cdot)$ takes the form of a matrix $(u, v) = u^\top J v$. There is a Lie algebra automorphism $\sigma : \g \ra \g$ of order 2 defined by $$\sigma(X)  = -J^{-1} X^\top J$$ which is independent of our choice of basis. Then $\h$ coincides with the $+1$ eigenspace of $\sigma$. The $-1$ eigenspace shall be denoted $\pp$. We have $\g = \h \oplus \pp$ and $\pp$ is a $K$-module. 

If $x \in \g$ then we may identify $\g_x$ with $\Lie(G_x)$ \cite[Theorem~2.5]{Jan}. The dual space to $\g_x$ will be denoted $\g_x^\ast$. If furthermore $x\in \h$ then we may identify $\h_x$ with $\Lie(K_x)$ for the same reason and, since $\sigma(x) = x$, the involution $\sigma$ induces a decomposition $\g_x = \h_x \oplus \pp_x$ which is $K_x$-stable. The dual space $\h_x^\ast$ identifies (as a $K_x$-module) with the annihilator of $\pp_x$ in $\g_x$, whilst $\pp_x^\ast$ identifies with the annihilator of $\h_x$ in $\g_x^\ast$. By duality we identify the symmetric algebra $S(\g_x)$ with $\K[\g_x^\ast]$ as $G_x$-modules, and identify $S(\h_x)$ and $S(\pp_x)$ with $\K[\h_x^\ast]$ and $\K[\pp_x^\ast]$ as $K_x$-modules. Using these identifications we obtain restriction maps $S(\g_x) \ra S(\h_x)$ and $S(\g_x) \ra S(\pp_x)$ which are $K_x$-module homomorphisms. The $p^\th$ power subalgebras shall be denoted $\K[\g_x^\ast]^p, \K[\h_x^\ast]^p, \K[\pp_x^\ast]^p$ respectively.

\begin{thm}\label{main1} Suppose that $\epsilon = -1$ so that $K$ is of type $C$. If $\tG \in \{G, K\}$ is of rank $l$ and $x \in \tg = \Lie(\tG)$ then
\begin{enumerate} 
\item{$\K[\tg_x^\ast]^{\tg_x}$ is free of rank $p^{l}$ over $\K[\tg^\ast_x]^p$;}
\smallskip
\item{$\K[\tg_x^\ast]^{\tG_x}$ is a polynomial algebra on $l$ generators;}
\smallskip
\item{$\K[\tg^\ast_x]^{\tg_x} \cong \K[\tg^\ast_x]^p \otimes_{(\K[\tg^\ast_x]^p)^{\tG_x}} \K[\tg_x^\ast]^{\tG_x}$.}
\end{enumerate}
\end{thm}

It is easily seen that the above theorem has a straightforward reduction to the nilpotent case, for if $x = x_s + x_n$ is the Jordan decomposition of $x \in \tg$ then $\tg_x = (\tg_{x_s})_{x_n}$. But $\tg_{x_s}$ is a direct sum of the centre and of simple groups of types $A$ or $C$, whilst the reductive rank of $\tg$ is equal to that of $\tg_{x_s}$. An inductive argument may then be used. 

It will be convenient to discuss Lie algebras in some generality, and so we let $\tg$ be an arbitrary Lie algebra over $\K$ and let $W$ be a $\tg$-module. The \emph{index of $\tg$ in $W$} is defined as
$$\ind(\tg, W) = \min_{\alpha\in W^\ast} \dim \tg_\alpha.$$
Note that the minimum is taken over elements of the dual space. The definition is due to Dixmier, who introduced it due to its importance in representation theory. It has subsequently been studied in many, many articles (eg. \cite{TY}, \cite{Yak1}, \cite{CMM}, \cite{Pan} to mention just a few). The most important instance of index is that of a Lie algebra in itself: we denote $\ind(\tg) := \ind(\tg, \tg)$ and call this quite simply the \emph{index of $\tg$}. A well known conjecture in Lie theory, now a theorem, is the Elashvili conjecture which states that for a reductive Lie algebra $\tg$ we have $\ind(\tg) = \ind(\tg_x)$ for all $x \in \tg^\ast$. The classical cases were settled by Yakimova in \cite{Yak1} whilst an almost general proof was found by Moreau and Charbonnel \cite{CMM} (the remaining seven cases being settled manually). This theorem was essential to \cite{PPY}, and so we too are indebted to this result. 

The proof in type $A$ and $C$ of Premet's conjecture in \cite{PPY} reduced to the nilpotent case and makes use of a deep understanding of Poisson structure on the coordinate algebra of a Slodowy slice $\C[e + \tg_f]$. They constructed invariants for the centraliser using well known invariants for the reductive group, and made a precise conjecture describing these invariants \cite[Conjecture~4.1]{PPY}. This was later confirmed in \cite[Theorem~15]{Yak2}. The proof in type $A$ in \cite{BB} was very different, and made use of a realisation of finite $W$-algebras as shifted Yangians. They then identified the invariants as the top graded components of the ``quantum determinants" \cite[$\S 2$]{BB}. These two descriptions of invariants actually coincide, thanks to \cite[$\S 6$]{Yak2}. 

As Theorem~\ref{main1} reduces to the nilpotent case we now choose $e \in \g$ nilpotent. Our approach is quite different from the two above and takes, as a starting point, the explicit presentations of the type $A$ characteristic 0 invariants discussed in the previous paragraph (we shall recall them in formula~(\ref{xrreform}) of Section~\ref{elemInv}). We then show that their reductions modulo $p$, which we denote by $x_1,x_2,...,x_N \in S(\g_e)$, are invariant under $G_e$ (Corollary~\ref{invariants}). Our method is to demonstrate that the $x_r$ satisfy a theorem of S. Skryabin \cite[Theorem~5.4]{Skr}, which tells us that Theorem~\ref{main1} will follow provided the Jacobian locus of the $x_1,...,x_N$ has codimension 2 in $\g_e^\ast$. This is demonstrated in Section~\ref{JacLoc1} using a scheme of argument not too dissimilar to \cite[$\S 3$]{PPY}.

Our proof in type $C$ is quite similar. In this case the invariants we consider are the restrictions $\{x_{2r}|_{\h_e^\ast}: r=1,...,\frac{N}{2}\}$ (the $x_i$ with $i$ odd restrict to zero on $\h_e^\ast$ by Corollary~\ref{vanishing1} which may be seen as a generalisation of the well known behaviour in case $e = 0$). Once again we must show that the invariants have a Jacobian locus of codimension 2. This property is actually inherited from the type $A$ case using an elementary yet original line of reasoning (Section~\ref{JacLoc2}). We shall state Skryabin's theorem in its full generality in Section~\ref{Proofs}.

As mentioned previously, our argument in type $C$ does not extend to types $B$ and $D$. The reason for this is quite simply that the restrictions $\{x_{2i}|_{\h_e^\ast}: i=1,...,\frac{N}{2}\}$ often fail to be algebraically independent. On the other hand, this method may be applied to some advantage when examining the invariants $\K[\pp_e^\ast]^{\h_e}$ in orthogonal cases:
\begin{thm}\label{main3}
Let $\epsilon = 1$ so that $K$ is of type $B$ or $D$. If $e \in h$ is nilpotent, with associated partition $\lambda$ and $m := (N + |\{ i : \lambda_i \text{ odd}\}|)/2$. Then
\begin{enumerate}
\item{$\K[\pp_e^\ast]^{\h_e}$ is free of rank $p^{m}$ over $\K[\pp_e^\ast]^p$;}
\smallskip
\item{$\K[\pp_e^\ast]^{K_e}$ is a polynomial algebra on $m$ generators;}
\smallskip
\item{$\K[\pp_e^\ast]^{\h_e} \cong \K[\pp_e^\ast]^p \otimes_{(\K[\pp_e^\ast]^p)^{K_e}} \K[\pp_e^\ast]^{K_e}$.}
\end{enumerate}
\end{thm}

We shall briefly discuss the method. Denote by $d_r$ the degree of $x_r$. The degrees $d_r$ have a very combinatorial description given in formula~(\ref{degreeinst}) of the next section. The invariants we consider are $\{x_{r}|_{\pp_e^\ast} : r + d_r \text{ even}\}$; the other invariants restrict to zero by Corollary~\ref{vanishing1}. Once again the codimension 2 condition on the Jacobian locus is inherited from the type $A$ case. Since that part of the argument is so similar to the type $C$ case we simply make a brief list of necessary changes to the proof. For Theorem~\ref{main1}, we have the correct number of invariants to apply Skryabin's theorem thanks to the confirmation of the Elashvili conjecture. For Theorem~\ref{main3}, however, the index $\ind(\h_e, \pp_e)$ is not known and we make a detour in Section~\ref{GenEl} to calculate this index by exhibiting the existence of a generic stabaliser. Some other results are obtained regarding indexes and generic stabalisers. The $e = 0$ case of the above theorem was proved in \cite[Theorem~0.22]{Lev} using quite similar methods.

Our remaining results concern the centres of enveloping algebras. Resume the setting of Theorem~\ref{main1} so that $\tG$ is a simple algebraic group of rank $l$ of type $A$ or $C$, and $x \in \tg$. Then $\tg_x$ is a $p$-Lie algebra (in the sense of \cite{Jac}) with $p$-operation $v \ra v^{[p]}$ induced by matrix multiplication. We let $U(\tg_x)$ denote the enveloping algebra of $\tg_x$. Our next theorem offers a description $U(\tg_x)^{\tG_x}$ and $U(\tg_x)^{\tg_x}$ analogous to the one given in Theorem~\ref{main1}. Let $Z(\tg_x)$ denote the centre of $U(\tg_x)$ and $Z_p(\tg_x)$ the $p$-centre: the central subalgebra of $U(\tg_x)$ generated by expressions of the form $v^p - v^{[p]}$ with $v \in \tg_x$. It is evident that $U(\tg_x)^{\tg_x} = Z(\tg_x)$.

We consider the canonical filtration $$U(\tg_x) = \bigcup_{i\geq 0} U(\tg_x)_{(i)}$$ where $U(\tg_x)_{(i)} = \K$. By the Poincare-Birkoff-Witt theorem the associated graded algebra is $\gr U(\tg_x) \cong S(\tg_x)$ and the corresponding grading coincides with the usual one on $S(\tg_x) = \K[\tg_x^\ast]$. We let $S(\tg_x)_i$ denote the vector subspace of $S(\tg_x)$ of homogeneous elements of degree $i$. Since each graded $\tG_x$-algebra is trivially a filtered $\tG_x$-algebra both $U(\tg_x)$ and $S(\tg_x)$ can be seen as filtered $\tG_x$-modules. A key ingredient in the proof of our next theorem is the existence of a filtration preserving isomorphism of $\tG_x$-modules $$\beta : U(\tg_x) \overset{\sim}\longrightarrow S(\tg_x)$$ such that the associated graded map $\gr (\beta)$ is the identity on $S(\tg_x)$. The isomorphism is known to exist for a large class of algebras \cite[$\S 3$]{Pre} and we shall describe it explicitly in Section~\ref{milnersec}. It is worth mentioning that such an isomorphism exists if and only if the natural inclusion of a Lie algebra into its enveloping algebra splits relative to the adjoint action \cite[Theorem~1.2]{FP}, although this is not always the case as shown by the example due to Wilson, presented following Theorem~1.4 of \cite{FP}. This sits in stark contrast to the characteristic zero case, when the symmetric and enveloping invariant subspaces are actually isomorphic as algebras.

We obtain generators for $Z(\tg_x)$ as the inverse images under $\beta$ of the aforementioned generators for $S(\tg_x)^{\tG_x}$, and then use a standard filtration argument to prove the next theorem. We remind the reader that when $\tG$ is of type $C$ we make the assumption $\chr(\K) > 2$.
\begin{thm}\label{main2}
\begin{enumerate}  Suppose that $\epsilon = -1$ so that $K$ is of type $C$. If $\tG \in \{G, K\}$ is of rank $l$ and $x \in \tg = \Lie(\tG)$ then
\item{$Z(\tg_x)$ is free of rank $p^l$ over $Z_p(\tg_x)$;}
\smallskip
\item{$U(\tg_x)^{\tG_x}$ is a polynomial algebra on $l$ generators;}
\smallskip
\item{$Z(\tg_x) \cong Z_p(\tg_x) \otimes_{(Z_p(\tg_x))^{\tG_x}} U(\tg_x)^{\tG_x}$.}
\end{enumerate}
\end{thm}

Our final results are representation theoretic in nature. In his seminal paper \cite{Zas} Zassenhaus observed that there is an upper bound on the dimensions of simple modules for any Lie algebra $\tg$ defined over a field of positive characteristic. We denote that upper bound by $M(\tg)$. A very coarse estimate is given by $M(\tg) \leq p^{\dim(\tg)}$ (see \cite[A.4]{Jan.} for example). Let $F(\tg)$ denote the full field of fractions of $Z(\tg)$ and let $F_p(\tg)$ denote that of $Z_p(\tg)$. Clearly $F(\tg)$ is a finite field extension of $F_p(\tg)$ and the degree (dimension of $F(\tg)$ over $F_p(\tg)$) is a power of $p$. Let us denote the degree by $[F(\tg) : F_p(\tg)] = p^l$. One of the many interesting consequences of Zassenhaus' aforementioned paper is that $M(\tg) = p^{\frac{1}{2}(\dim(\tg) - l)}$ \cite[A.5]{Jan.}. Kac and Weisfeiler made a series of conjectures about the dimensions of simple modules. The first of these (KW1) states that $$M(\tg) = p^{\frac{1}{2}(\dim(\tg) - \ind(\tg))}.$$ In light of Zassenhaus' contribution it suffices to show that $[F(\tg) : F_p(\tg)] = p^{\ind(\tg)}$ in order to prove the conjecture for $\tg$. This is precisely our approach in proving our fourth main theorem.
\begin{thm}\label{KW1}
If $\tG$ is of type $A$ or $C$ and $x \in \tg$ then the KW1 conjecture holds for $\tg_x$.
\end{thm}
The conjecture stated in its full generality is quite unnerving, and a full proof is a long way off. That being said, there are no known counterexamples. The Jacobson-Witt algebras form an interesting class to discuss here. The algebra $W(n; \bf{1})$ is defined as the derivation algebra of the truncated polynomial algebra $\K[t_1, ..., t_n]/(t_1^p, ...,t_n^p)$. It is simple of rank $n$ and non-algebraic. The conjecture is settled in the affirmative for $W(1; \bf{1})$ and $W(2; \bf{1})$ \cite{PS} using a powerful theorem stating that KW1 holds for $\tg$ provided there exists $v \in \tg^\ast$ such that $\tg_v$ is a torus. In higher ranks very little is known, and as a result $W(3; \bf{1})$ is believed to be a solid proving ground for the conjecture \cite[$\S 4.1$]{PS}.

In finite characteristic $Z(\tg)$ is a finitely generated, commutative integral domain. Therefore we may consider the algebraic variety $\Zas(\tg) := \Specm(Z(\tg))$, known as the Zassenhaus variety \cite{Zas}, \cite{PT}. The geometry of $\Zas(\tg)$ is of some importance and controls certain algebraic features of $\tg$. Theorem~\ref{KW1} allows us to make some fine deductions on the singular locus of $\Zas(\tg_x)$ when $\tG$ is of type $A$ or $C$ and $x \in \tg$.  We shall reserve the full discussion for the final section.

\renewcommand{\abstractname}{Acknowledgements}
\begin{abstract}
I would like to thank Alexander Premet for his excellent guidance throughout this research. I would also like to thank Oksana Yakimova and Anne Moreau for taking the time to read early drafts of this article and for their many helpful suggestions.
\end{abstract}

\section{The Elementary Invariants}\label{elemInv}

Retain the notations and conventions of the introduction so that $G = GL(V)$, etc. Throughout this section we fix nilpotent $e \in \g$. Then $V$ decomposes into minimal $e$-stable subspaces: $V = \oplus_{i=1}^n V[i]$. The spaces $V[i]$ are called the Jordan block spaces of $e$, and $e$ acts as a regular nilpotent endomorphism on each of them. Define $\lambda_i := \dim V[i]$. We may assume $\lambda_1 \geq \cdots \geq \lambda_n$, so that $\lambda := (\lambda_1,...,\lambda_n)$ is an ordered partition of $N$. The nilpotent $G$-orbits are classified by their partitions in this way. For $i = 1,...,n$ we may choose vectors $w_i \in V[i] \backslash \im(e)$ so that $\{e^s w_i\}_{s=0}^{\lambda_i - 1}$ is a basis for the Jordan block space $V[i]$ and $\{e^s w_i : 1 \leq i \leq n, 0 \leq s < \lambda_i\}$ is a basis for $V$.

Let $\xi \in \mathfrak{g}_e$. Then $\xi (e^s w_i) = e^s (\xi w_i)$ showing that $\xi$ is determined by its action on the $w_i$. If we define
\begin{eqnarray*}
\xi_i^{j,s} w_k = \left\{ \begin{array}{ll}
         e^sw_j & \mbox{ if $i=k$}\\
        0 & \mbox{ otherwise}\end{array} \right.
\end{eqnarray*}
and extend the action to $\{e^s w_i\}$ by the requirement that $\xi_i^{j,s}$ is linear and centralises $e$ then
\begin{eqnarray*}
\{\xi_i^{j,\lambda_j - 1- s} : 1 \leq i,j \leq n , 0  \leq s < \min(\lambda_i, \lambda_j)\}
\end{eqnarray*}
forms a basis for $\mathfrak{g}_e$. By convention, if either $i$, $j$ or $s$ lie outside the specified range then we take $\xi_i^{j,\lambda_j-1-s} = 0$. Let $\mathfrak{g}_e^\ast$ denote the dual space of $\mathfrak{g}_e$. We define a dual basis $\{(\xi_i^{j,s})^\ast\}$ for $\mathfrak{g}_e^\ast$ by setting
\begin{eqnarray*}
(\xi_i^{j,s})^\ast(\xi_k^{l,r}) = \left\{ \begin{array}{ll}
         1 & \mbox{ if $i=k$, $j=l$, $s=r$}\\
        0 & \mbox{ otherwise}\end{array} \right.
\end{eqnarray*}

As $\lambda = (\lambda_1,...,\lambda_n)$ is a partition of $N$ we may define a \emph{composition of} $\lambda$ to be a finite sequence $\mu = (\mu_1,...,\mu_n)$ with $0 \leq \mu_i \leq \lambda_i$ for $i = 1,...,n$. We write $|\mu| = \sum \mu_i$ and let $l(\mu)$ denote the number of $i$ for which $\mu_i$ is nonzero. If $1 \leq j \leq l(\mu)$ then $i_j^{\mu}$ denotes the $j^{\text{th}}$ index such that $\mu_{i_j} \neq 0$ (ordered so that $i_1^{\mu} \leq i_2^{\mu} \leq i_3^\mu \leq \cdots$). In our case, a choice of $\mu$ shall always be clear from the context, and we will usually write $i_j = i_j^{\mu}$.

With a view to introducing the elementary symmetric invariants described in \cite{PPY} and \cite{BB} we define the sequence of invariant degrees by
\begin{eqnarray}\label{degreeinst}
(d_1, ..., d_N) = (\overbrace{1,1,...,1}^{\lambda_1 \text{ 1's}},\overbrace{2,...,2}^{\lambda_2 \text{ 2's}},...,\overbrace{n,...,n}^{\lambda_n \text{ n's}})
\end{eqnarray}
Fix $r \in \{1,...,N\}$ and set $d = d_r$. Denote by $\mathcal{C_\lambda}$ the set of compositions of $\lambda$ fulfilling $|\mu| = r, l(\mu) = d$ and let $\gS_d$ denote the symmetric group on $d$ letters. If a choice of $w \in \gS_d$ and a composition $\mu \in \mathcal{C}_\lambda$ is implicit and $1\leq j \leq l(\mu)$ then we write
\begin{eqnarray*}
s_j = s_j(w,\mu) := \lambda_{i_{wj}^{\mu}} - \lambda_{i_j^{\mu}} + \mu_{i_j^{\mu}} -1
\end{eqnarray*}
Thanks to the identification of $S(\g_e)$ and $\K[\g_e^\ast]$, the latter is spanned by monomials in the linear unary forms $\xi_i^{j,s} : \g_e^\ast \ra \K$ and we define
\begin{eqnarray*}
&\Theta_r : \gS_{d_r} \times \mathcal{C}_\lambda \ra \K[\g_e^\ast];\\
&\Theta_r (w, \mu) = \sgn(w)\xi_{i_1}^{i_{w1}, s_1} \cdots \xi_{i_d}^{i_{wd},s_d}.
\end{eqnarray*}
Now the our elementary invariants may be defined
\begin{eqnarray}\label{xrreform}
x_r = \sum_{(w,\mu) \in \gS_d\times \mathcal{C}_\lambda} \Theta_r(w,\mu) \in \K[\g_e^\ast]. 
\end{eqnarray}
Each $x_r$ is homogeneous of degree $d_r$.

\begin{rem}
\rm{Let's review a brief history of these polynomials. The definition of $x_r$ first appeared in \cite[Conjecture~4.1]{PPY} over $\C$, where it was conjectured that this would be the explicit formula for the invariants being studied there, which were denoted ${}^eF_1,...,{}^eF_N$. The same invariants were derived using different methods in \cite{BB}, defined over a different basis, but presented in our notation $x_1,...,x_N$. In \cite[$\S 6$]{Yak2} these bases were shown to coincide so \cite[Conjecture~4.1]{PPY} is equivalent to ${}^eF_i = x_i$, and in \cite[Theorem~15]{Yak2} the conjecture was confirmed.
}\end{rem}
\begin{rem}
\rm{It is instructive to unfold the definition of $x_r$ for certain nilpotent orbits $e$. First we shall take the case $e = 0$, so that $\lambda = (1,1,...,1)$ and $\g_e = \g$. Fix $r = 1,...,N$. Then $\mathcal{C}_\lambda$ is the set of subsets of size $r$ of $\{1,...,N\}$. The basis vectors $\xi_i^{j,0}$ are now just the elementary matrix units. It follows that $x_r$ is dual to the $r^\th$ coefficient of the characteristic polynomial, ie. the classical Harish-Chandra invariants; see \cite[Theorem~1.4.1]{Smi} for a construction via elementary methods. 

Now consider the case $e$ regular, with partition $(N)$. It is well known that $\g_e$ is abelian. More precisely \cite[Theorem~4]{Yak2} tells us that $\g_e$ is spanned by matrix powers of $e$. Fix $r=1,...,N$, observe that $d_r = 1$ and that $\mathcal{C}_\lambda$ contains a single element $(r)$. Furthermore $\Sf_d$ is trivial. Hence $x_r = \xi_1^{1, r-1} = e^{r-1}$ and the invariants are a basis for $\g_e$, as is to be expected.

For the orbits lying between these two extremal cases, the invariants are impossible to describe in a uniform manner, although we may think of them as a generalisation of the above two examples.}
\end{rem}
Note that all groups, spaces and maps discussed in the current section may equally well be defined over $\C$. The following result is a consequence of \cite[Theorem~4.2]{PPY}, or indeed of \cite[Theorem~4.1]{BB}.
\begin{thm}\label{BrB}
The polynomials $x_1,...,x_N$ are $G_e$-invariant when defined over $\C$.
\end{thm}
We use reduction modulo $p$ to obtain the following.
\begin{cor}\label{invariants}
The polynomials $x_1,...,x_N$ are $G_e$-invariant when defined over $\K$.
\end{cor}
\begin{proof}
We proceed with all maps and spaces defined over $\C$. The matrix $\xi_i^{j,s}$ is nilpotent when $i \neq j$ or when $s>0$. The elements $\xi_i^{i,0}$ are semisimple. We shall denote by $M_N(\C)$ the ring of $N\times N$ matrices over $\C$ with the usual associative multiplication and basis given the standard matrix units with respect to our $V$-basis $\{e^s w_i : 1\leq i \leq n, 0 \leq s < \lambda_i\}$. Denote by $M_N(\Z)$ the $\Z$-lattice in $M_N(\C)$ arising from the canonical inclusion $\Z \subseteq \C$ and from our choice of $\C$-basis for $M_N(\C)$. Let $X$ be an indeterminate over $\C$ and denote by $M_N(\C)[X]$ the ring of polynomials in $X$ with coefficients in $M_N(\C)$. Let $i \neq j$, let $\lambda_j - \min(\lambda_i, \lambda_j) \leq s < \lambda_j$ and $0 < r < \lambda_i$. Consider the matrices
\begin{eqnarray}\label{gh}
g_{X} = 1 + X\xi_i^{j,s} \text{ and } h_{X} = 1 + X\xi_i^{i,r}.
\end{eqnarray}
They both lie in $M_N(\C)[X]$ and the reader may verify that their inverses are
\begin{eqnarray}\label{inv}
g_{X}^{-1} = 1 - X\xi_i^{j,s} \text{ and } h_X^{-1} = 1 + \sum_{k=1}^{\infty} (-1)^k X^k \xi_i^{i,kr}.
\end{eqnarray}
Note that $M_N(\Z)[X]$ canonically includes into $M_N(\C)[X]$ as a $\Z$-lattice. Since each $\xi_i^{j,s}$ is actually an element of $M_N(\Z)$ and all coefficients in (\ref{gh}) and (\ref{inv}) are integral we may consider the above maps to be elements of $M_N(\Z)[X]$. The elements $x_r$ are also elements of the symmetric algebra $S(M_N(\Z))$, so we get $\Ad(g_X) x_r \in S(M_N(\Z)[X])$. Hence the expression $\Ad(g_X) x_r$ can be written $A_{0,r} + X A_{1,r} + X^2 A_{2,r} + \cdots$ where $A_{i,r} \in S(M_N(\Z))$. By theorem \ref{BrB} we see that $x_r$ is $g_X$-stable for any $X \in \C$ which means $A_{0,r} = x_r$ and $A_{i,r} = 0$ for all $i > 0$. Let $\tilde{A}_{i,r} = A_{i,r} \otimes 1 \in S(M_N(\Z)[X]) \otimes_{\Z} \K$. Then for each $r$ we still have $\tilde{A}_{0,r} = x_r \otimes 1 \in S(M_N(\Z))\otimes_{\Z} \K$ and $\tilde{A}_{i,r}= 0 $ for $i > 0$. But $x_r \otimes 1$ is just $x_r$ defined over $\K$. Now we define $\tilde{g}_X$ analogously to $g_X$, except we now take our maps to be defined over $\K$ and allow $X$ to vary over $\K$. Then the reader may check quite easily that $$\Ad(\tilde{g}_X)(x_r \otimes 1) = \sum_{i\geq 0} \tilde{A}_{i,r}X^i = \tilde{A}_{0,r} = x_r \otimes 1.$$ We may define $\tilde{h}_X$ to be the the analogue of $h_X$ over $\K$ and the same line of reasoning will show that $\Ad(\tilde{h}_X)(x_r\otimes 1) = x_r \otimes 1$.

The remnant of the discussion shall be conducted entirely over $\K$. Let $U$ denote the subgroup of $G_e$ generated by the lines $\{1 + X\xi_i^{j,s}: X \in \K\}$ and $\{1 + X \xi_i^{i,r}: X \in \K\}$ where $i,j,s$ and $r$ vary in the range $i \neq j$, $\lambda_j - \min(\lambda_i, \lambda_j) \leq s < \lambda_j$ and $0 < r < \lambda_i$. The above shows that the lines generating $U$ fix each $x_r$ over $\K$; so must the closure $\overline{U}$.

Consider the group $T \subset G$ in which the elements act on each $V[i]$ $(i = 1,...,n)$ by an arbitrary nozero scalar. The group is toral, closed, contained in $G_e$ and $\Lie(T) = \spn\{ \xi_i^{i,0} : 1\leq i \leq n\}$. We shall show that $T$ fixes each $x_r$. Recall that $\xi_{i_j}^{i_{wj}, s_j} \in \Hom(V[i_j], V[i_{wj}])$ so if $h \in T$ is defined by $hw_i = c_i w_i$ for scalars $c_i \in \K$ then
\begin{eqnarray*}
\Ad(h)\xi_{i_j}^{i_{wj}, s_j} = (c_{i_{j}}^{-1}c_{i_{wj}}) \xi_{i_j}^{i_{wj}, s_j}
\end{eqnarray*}
and
\begin{eqnarray*}
\Ad(h)x_r = \sum_{(w,\mu) \in \gS_d \times \mathcal{C}_\lambda}  \sgn(w)\overbrace{(\prod_{j=1}^dc_{i_{j}}^{-1}c_{i_{wj}})}^{=1} \xi_{i_1}^{i_{w1}, s_1} \cdots \xi_{i_d}^{i_{wd}, s_d} = x_r
\end{eqnarray*}
Now the group $\langle \overline{U}, T \rangle$ generated by $\overline{U}$ and $T$ is closed, connected and its Lie algebra contains a basis for $\g_e$. Hence $\langle \overline{U}, T \rangle = G_e$. Both $\overline{U}$ and $T$ fix $x_r$; so must $G_e$.
\end{proof}

\section{The Restriction of Invariants to Classical Subalgebras}\label{restriction}

Define $K$, $\h$, $\pp$, $(\cdot,\cdot)$, $\epsilon$ and $\sigma$ as per the introduction. Recall that char$(\K) > 2$ when we discuss $K$. Let $e \in \h$ be a nilpotent element and, as usual, denote the Jordan block sizes by $\lambda_1 \geq \lambda_2 \geq \cdots \geq \lambda_n$. As explained in the previous section there exist vectors $\{w_i\}$ such that $\{e^s w_i : 1 \leq i \leq n, 0 \leq s < \lambda_i\}$ forms a basis for $V$. The following condition on the Jordan block sizes can be found in \cite[Theorem~1.4]{Jan}.
\begin{lem}\label{nilpotents}
The $w_i \in V$ can be chosen so that there exists an involution $i \mapsto i'$ on the set $\{1,...,n\}$ such that
\begin{enumerate}
\item{$\lambda_i = \lambda_{i'}$ for all $i = 1,...,n$}
\smallskip
\item{$i = i'$ if and only if $\epsilon (-1)^{\lambda_i} = -1$}
\end{enumerate}
\end{lem}

In plain English, the lemma implies that for a nilpotent element in a symplectic algebra, each Jordan block of odd dimension can be paired with a different Jordan block of the same dimension, and in an orthogonal algebra each Jordan block of even dimension can be paired with a different Jordan block of the same dimension. It is well known that every partition fulfilling these criteria corresponds to a nilpotent $K$-orbit in $\h$. Without disturbing the ordering on the $\lambda_i$ we may (and shall henceforth) assume that $i' \in  \{i-1,i,i+1\}$. Before we continue it will be convenient to normalise the basis for $V$, following the procedure of \cite{Yak2}. Let $\{w_i\}$ be chosen in accordance with the above lemma. Fix $i$. We have $(e^s w_i, w_{i'}) = (-1)^s (w_i, e^s w_{i'})$ so $e^{\lambda_i-1}w_i$ is orthogonal to all $e^sw_{i'}$ with $s > 0$. There is a (unique upto scalar) vector $v \in V[i]$ which is orthogonal to all $e^s w_{i'}$ for $s < \lambda_i - 1$. This $v$ does not lie in $\im(e)$ for otherwise it would be othogonal to all of $V[i] + V[i']$. This of course is not possible since the restriction of $(\cdot,\cdot)$ to $(V[i] + V[i']) \oplus (V[i] + V[i'])\subseteq V \oplus V$ is non-degenerate. It does no harm to replace $w_i$ by $v$ and normalise according to the rule
\begin{eqnarray*}
(w_i, e^{\lambda_i - 1} w_{i'}) = 1 & \text{ whenever } i \leq i'
\end{eqnarray*}

With respect to this basis the matrix of the restriction of $(\cdot,\cdot)$ to $V[i] + V[i']$ is antidiagonal with entries $\pm 1$. We now describe a spanning set for $\h_e$. We have
\begin{eqnarray}\label{sigmaaction}
\sigma(\xi_i^{j,\lambda_j - 1 - s}) = \varepsilon_{i,j,s}\xi_{j'}^{i', \lambda_i - 1 - s}
\end{eqnarray}
where $\varepsilon_{i,j,s}$ is defined by the following relationship \cite{Yak2}
\begin{eqnarray*}
(e^{\lambda_j-1-s}w_j,e^sw_{j'}) = - \varepsilon_{i,j,s}(w_i, e^{\lambda_i-1} w_{i'})
\end{eqnarray*}
Make the notations
\begin{eqnarray*}
\varpi_{i \leq j} = \left\{ \begin{array}{ll}
         1 & \mbox{ if $i\leq j$}\\
        -1 & \mbox{ if $i > j$}\end{array} \right.\\
\zeta_i^{j,s} =  \xi_i^{j,\lambda_j-1-s} +  \varepsilon_{i,j,s}\xi_{j'}^{i', \lambda_i - 1 - s}\\
\eta_i^{j,s} = \xi_i^{j,\lambda_j-1-s} -  \varepsilon_{i,j,s}\xi_{j'}^{i', \lambda_i - 1 - s}.
\end{eqnarray*}
The maps $\zeta_i^{j,s}$ span $\h_e$ and the maps $\eta_i^{j,s}$ span $\pp_e$. We define a dual spanning set
\begin{eqnarray*}
(\zeta_i^{j,s})^\ast := (\xi_i^{j,\lambda_j-1-s})^\ast +  \varepsilon_{i,j,s}(\xi_{j'}^{i', \lambda_i - 1 - s})^\ast\\
(\eta_i^{j,s})^\ast := (\xi_i^{j,\lambda_j-1-s})^\ast -  \varepsilon_{i,j,s}(\xi_{j'}^{i', \lambda_i - 1 - s})^\ast.
\end{eqnarray*}
Note that we do \emph{not} have $(\zeta_i^{j,s})^\ast(\zeta_k^{l,r}) \in \{0,1\}$, contrary to the common convention for dual basis vectors. For instance, $\zeta_i^{i,\lambda_i-2} = 2\xi_i^{i,1}$ when $i = i'$ and $\lambda_i \geq 2$, and so $(\zeta_i^{i,\lambda_i-2})^\ast(\zeta_i^{i,\lambda_i-2}) = 4$. Our discussion shall be more concise due to this choice of definition (see Remark~\ref{duals} below).

\begin{lem}\label{spanningdetails}
The following are true:
\begin{enumerate} 
\item{$\varepsilon_{i,j,s} = (-1)^{\lambda_j - s} \varpi_{i\leq i'} \varpi_{j \leq j'}$;}
\smallskip
\item{$\varepsilon_{i,j,s} = \varepsilon_{j',i',s}$;}
\smallskip
\item{The only linear relations amongst the $\zeta_i^{j,s}$ are those of the form $\zeta_i^{j,s} = \varepsilon_{i,j,s} \zeta_{j'}^{i',s}$;}
\smallskip
\item{The only linear relations amongst the $\eta_i^{j,s}$ are those of the form $\eta_i^{j,s} = - \varepsilon_{i,j,s} \eta_{j'}^{i',s}$;}
\smallskip
\item{Analogous relations hold amongst the $(\zeta_i^{j,s})^\ast$ and the $(\eta_i^{j,s})^\ast$. These are the only such relations;}
\smallskip
\item{The $(\zeta_i^{j,s})^\ast$ span $\h_e^\ast$ and the $(\eta_i^{j,s})^\ast$ span $\pp_e^\ast$.}
\end{enumerate}
\end{lem}
\begin{proof}
We have $$\varepsilon_{i,j,s} = \frac{-(e^{\lambda_j-1-s}w_j,e^sw_{j'})}{(w_i, e^{\lambda_i-1} w_{i'})} = \frac{(-1)^{\lambda_j-s}(w_j,e^{\lambda_j-1}w_{j'})}{(w_i, e^{\lambda_i-1} w_{i'})}.$$ We claim that $(w_i, e^{\lambda_i - 1} w_{i'}) = \varpi_{i \leq i'}$. The bilinear form $(\cdot,\cdot)$ is normalised so that $(w_i, e^{\lambda_i - 1} w_{i'}) = 1$ whenever $i \leq i'$, therefore we need only show that $(w_i, e^{\lambda_i - 1} w_{i'}) = -1$ whenever $i > i'$. If $i > i'$ then $$(w_i, e^{\lambda_i - 1} w_{i'}) = (-1)^{\lambda_i - 1} (e^{\lambda_i-1} w_i, w_{i'}) = \epsilon(-1)^{\lambda_i-1} (w_{i'}, e^{\lambda_i-1}w_i) =  \epsilon(-1)^{\lambda_i-1}.$$ However $ \epsilon(-1)^{\lambda_i-1} = -1$ whenever $i \neq i'$ by lemma \ref{nilpotents}, conculding part 1. Next observe that $\varpi_{i\leq i'} \varpi_{i' \leq i} = 1$ if and only if $i = i'$ hence $\varpi_{i\leq i'} \varpi_{i'\leq i} = \epsilon(-1)^{\lambda_i-1}$. Part 2 now follows from part 1.

The equality $\zeta_i^{j,s} = \varepsilon_{i,j,s} \zeta_{j'}^{i',s}$ follows from part 2. To see that these are the only relations we note that $\pp_e$ is spanned by vectors $\xi_i^{j,\lambda_j-1-s} - \varepsilon_{i,j,s} \xi_{j'}^{i',\lambda_i - 1-s}$ and that $\g_e / \pp_e \cong \h_e$. Then the map $ \xi_i^{j,\lambda_j-1-s} + \pp_e\ra \zeta_i^{j,s}$ is well defined and extends to a linear map $\g_e / \pp_e \ra \h_e$. It is surjective and so by dimension considerations  it is a vector space isomorphism. The only linear relations amongst the vectors $\xi_i^{j,\lambda_j-1-s} + \pp_e$ in $\g_e / \pp_e$ are those of the form $\xi_i^{j,\lambda_j-1-s} - \varepsilon_{i,j,s} \xi_{j'}^{i',\lambda_i - 1-s} + \pp_e= 0$. Part 3 follows. An identical argument works for part 4.

Part 5 is a straightforward consequence of duality, whilst for 6 we observe that $(\zeta_i^{j,s})^\ast$ vanishes on each $\eta_k^{l,r}$ and $(\eta_i^{j,s})^\ast$ vanishes on each $\zeta_k^{l,r}$. Since the set of all $(\zeta_i^{j,s})^\ast$ and all $(\zeta_i^{j,s})^\ast$ together span $\g_e^\ast$ it is quite evident that the zetas must span $\h_e^\ast$ and the etas must span $\pp_e^\ast$.
\end{proof}
\begin{rem}\label{duals}
\rm{Using part 3 of the above lemma it is quite easy to describe a basis for $\h_e$ by refining the set $\{\zeta_i^{j,s}\}$. This basis has a rather tedious combinatorial description which we shall not need, and so we avoid it here. From this basis it is possible to define a dual basis in the traditional manner, and it is easy to prove that this coincides (upto scalars) with the subset of $\{(\zeta_i^{j,s})^\ast\}$ which is obtained by carrying out the analogous refinement upon the set $\{(\zeta_i^{j,s})^\ast\}$ spanning $\h_e^\ast$. The point of this remark is that our definition of $(\zeta_i^{j,s})^\ast$ coincides with the more canonical one, although ours is easier to define. Analogous statements hold for $\pp_e$.}
\end{rem}

Fix $r \in \{1,...,N\}$ and set $d = d_r$. Recall that $\mathcal{C_\lambda}$ denotes the set of compositions of $\lambda$ fulfilling both $|\mu| = r$ and $l(\mu) = d$, and that $\Theta_r (w, \mu) = \sgn(w)\xi_{i_1}^{i_{w1}, s_1} \cdots \xi_{i_d}^{i_{wd},s_d}$ where $s_j = \lambda_{i_{wj}^{\mu}} - \lambda_{i_j^{\mu}} + \mu_{i_j^{\mu}} -1$. The invariants described in Section~1 are defined by
\begin{eqnarray}\label{xrreform}
x_r = \sum_{(w,\mu) \in \gS_d\times \mathcal{C}_\lambda} \Theta_r(w,\mu) \in \K[\g_e^\ast]^{G_e}. 	%Note that the choices of $\mu$ correspond to nilpotent orbits in certain Levis,
																			% and the permutations correspond to symmetries of those Levis. This could be
																			% another way of interpreting/defining the invariants, and could lead to a uniform
																			% construction for orthogonal types.
\end{eqnarray}
Since $\K[\g_e^\ast]$ is spanned by monomials in the basis $\{\xi_i^{j,s}\}$, the involution $\sigma : \g_e \ra \g_e$ extends uniquely to a $\K$-algebra automorphism of $\K[\g_e^\ast] \ra \K[\g_e^\ast]$ which we shall also denote by $\sigma$. This extension is clearly involutory. The following proposition shall be pivotal to our understanding of the symmetric invariants for centralisers of nilpotent elements in classical subalgebras of $\g$. 
\begin{prop}\label{sigmaxr}
$\sigma(x_r) = (-1)^r x_r$
\end{prop}
\begin{proof}
Fix $(w,\mu) \in \gS_{d} \times \mathcal{C}_\lambda$. We shall show that if $\Theta_r(w,\mu) \neq 0$ then there exists a pair $(w', \mu') \in \gS_d \times \mathcal{C}_\lambda$ such that $\sigma(\Theta_r(w,\mu)) = (-1)^{r}\Theta_r(w', \mu')$. In view of formula (\ref{xrreform}) the proposition shall follow. By formula (\ref{sigmaaction}) we have
\begin{eqnarray*}
\sigma(\Theta_r(w,\mu)) &=& \sgn(w)\sigma(\xi_{i_1}^{i_{w1}, s_1} \cdots \xi_{i_d}^{i_{wd}, s_d})\\ &=& (\prod_{k=1}^d\varepsilon_{i_k, i_{wk}, \lambda_{i_k} - \mu_{i_k}})\sgn(w) \xi_{(i_{w1})'}^{(i_{1})', \mu_{i_1} - 1} \cdots \xi_{(i_{wd})'}^{(i_{d})',\mu_{i_d} - 1}
\end{eqnarray*}
We must examine the coefficient 
\begin{eqnarray*}
\prod_{k=1}^d\varepsilon_{i_k, i_{wk}, \lambda_{i_k} - \mu_{i_k}} &=& \prod_{k=1}^d (-1)^{\lambda_{i_{wk}} - \lambda_{i_k} + \mu_{i_k}} \varpi_{i_k \leq i_k'} \varpi_{i_{wk} \leq i_{wk}'}\\ &=& (\prod_{k=1}^d (-1)^{\lambda_{i_{k}}} \varpi_{i_{k} \leq i_{k}'})(\prod_{k=1}^d (-1)^{\lambda_{i_{wk}}} \varpi_{i_{wk} \leq i_{wk}'})(\prod_{k=1}^d (-1)^{\mu_{i_k}})
\end{eqnarray*}
By definition $|\mu| = r$ implies $$\prod_{k=1}^d (-1)^{\mu_{i_k}} = (-1)^{\sum_{k=1}^d \mu_{i_k}} = (-1)^r.$$ Since $w$ is a permutation $$\prod_{k=1}^d (-1)^{\lambda_{i_{k}}} \varpi_{i_{k} \leq i_{k}'} = \prod_{k=1}^d (-1)^{\lambda_{i_{wk}}} \varpi_{i_{wk} \leq i_{wk}'}$$ and the product $\prod_{k=1}^d\varepsilon_{i_k, i_{wk}, \lambda_{i_k} - \mu_{i_k}}$ reduces to $(-1)^r$.

It remains to prove that $\Theta_r(w', \mu') = \sgn(w) \xi_{(i_{w1})'}^{(i_{1})', s_1+\lambda_{i_{1}} - \lambda_{i_{w1}}} \cdots \xi_{(i_{wd})'}^{(i_{d})',s_d+\lambda_{i_{d}} - \lambda_{i_{wd}}}$ for some choice of $(w', \mu') \in \gS_d\times \mathcal{C}_\lambda$. For $k = 1,...,d$ let $j_k = (i_{wk})'$; note that $j_1,...,j_d$ is not an ordered sequence. Let $\mu'$ be the composition of $\lambda$ with nonzero entries in positions indexed by the $j_k$ such that $\mu'_{j_k} = \mu_{i_k} + \lambda_{i_{wk}} - \lambda_{i_k}$. We have $l(\mu') = d$ by definition. Furthermore $$|\mu'| = \sum_{i=1}^k \mu'_{j_k} = \sum_{i=1}^k (\mu_{i_k} + \lambda_{i_{wk}} - \lambda_{i_k}) = \sum_{i=1}^k \mu_{i_k} = |\mu| = r.$$ Now in order to see that $0 \leq \mu'_k \leq \lambda_k$ we use the fact that $\Theta_r(w,\mu) \neq 0$. We have $0 \leq \lambda_{i_{wk}}-1-s_k \leq \min(\lambda_{i_k}, \lambda_{i_{wk}})$ and by definition of $s_k$ we have $$0 \leq \lambda_{i_k} - \mu_{i_k} \leq \lambda_{i_{wk}}.$$ The left hand inequality gives us $\mu'_{j_k} \leq \lambda_{i_{wk}} = \lambda_{j_k}$, whilst the right hand inequality tells us that $0 \leq \lambda_{i_{wk}} - (\lambda_{i_k} - \mu_{i_k}) = \mu'_{j_k}$. It follows that $0 \leq \mu'_k \leq \lambda_k$ for $k=1,...,n$ so that $\mu'$ is a composition of $\lambda$. We have shown that $\mu' \in \mathcal{C}_\lambda$.

We now aim to define $w'\in \gS_d$. If we define an element $\omega \in \mathfrak{S}_n$ by $\omega (i_k) = i_{wk}$ for $k = 1,...,d$ and $\omega( i ) = i$ for all $\mu_i = 0$, then $\sgn(\omega) = \sgn(w)$. The element $\nu$ of $\mathfrak{S}_n$ which sends $(i_{wk})'$ to $(i_k)'$ is $' \circ \omega^{-1} \circ '$, where $\circ$ denotes composition of permutations in $\mathfrak{S}_n$. But $'$ is an involution so $\nu$ and $\omega^{-1}$ are conjugate in $\mathfrak{S}_n$ implying $\sgn(\nu) = \sgn(\omega)$. Now $\nu$ preserves the set $\{j_1, ..., j_d\}$ and acts trivially on its complement, therefore we may take $w'$ to be the restriction of $\nu$ to $\{j_1, ..., j_d\}$. This element will have $\sgn(w') = \sgn(\nu) = \sgn(\omega) = \sgn(w)$.

From the above we get $j_{w' k} = \nu (i_{wk})' = (i_{k})'$, $\mu'_{j_{k}} = \mu'_{(i_{wk})'} = \mu_{i_k} + \lambda_{i_{wk}} - \lambda_{i_k}$, $\lambda_{j_{w'k}} = \lambda_{i_k}$ and $\lambda_{j_k} = \lambda_{i_{wk}}$ so that $$\xi_{j_k}^{j_{w'k}, \lambda_{j_{w'k}} - \lambda_{j_k} + \mu'_{j_k} - 1} = \xi_{(i_{wk})'}^{(i_{k})',\mu_{i_k} - 1}.$$ It follows that $(w', \mu')$ is the required pair. Since $\sigma$ is non-degenerate on $\g_e$ and $\Theta_r(w',\mu')$ is a product of terms $\sigma(\xi_{i_k}^{i_{wk}, s_k})$ with each $\xi_{i_k}^{i_{wk}, s_k} \neq 0$ we conclude that $\Theta_r(w',\mu') \neq 0$.
\end{proof}
As an immediate corollary we get

\begin{cor}\label{vanishing1}
\begin{enumerate} The following are true:
\item{$x_{r}|_{\mathfrak{k}_e^\ast} = 0$  for $r$ odd;}
\item{$x_{r}|_{\pp_e^\ast} = 0$ for $r + d_r$ odd.}
\end{enumerate}
\end{cor}
\begin{proof}
Let $\BB_0$ be a basis for $\h_e$ and $\BB_1$ a basis for $\pp_e$, so that $\BB_0 \cup \BB_1$ is a basis for $\g_e$. Then the monomials in $\BB_0 \cup \BB_1$ form an eigenbasis for the action of $\sigma$ on $\K[\g_e^\ast]$. The eigenvalues are $\pm1$ depending on the parity of the number of factors coming from $\BB_1$ in a given monomial. 

Now fix $r=1,...,N$ and write $x_r$ in the above eigenbasis. If $r$ is odd then by Proposition~\ref{sigmaxr} we have $\sigma x_r = - x_{r}$. It follows that there are an odd (ie. nonzero) number of factors from $\BB_1$ in each monomial summand of $x_r$. Each of these factors vanishes on $\h_e^\ast$; so must $x_r$, whence 1.

If $r + d_r$ is odd then there are two possibilities: either $r$ is odd and $d_r$ even or vice versa. Assume the former so that $x_r$ is a sum of monomials of even degree and $\sigma x_r = - x_r$. There must be an odd number of factors from $\BB_1$ in each monomial and since each monomial has even degree there is also an odd (ie. nonzero) number of factors from $\BB_1$ in each monomial. Each of these factors restrict to zero on $\pp_e^\ast$ and so must $x_r$. If, on the other hand, $r$ is even and $d_r$ is odd then $\sigma x_r = x_r$ so each monomial summand of $x_r$ contains an even number of factors from $\BB_1$. Since each such monomial has odd degree, there is an odd (ie. nonzero) number of factors from $\BB_0$ in each monomial. Again each of these factors restrict to zero on $\pp_e^\ast$, and so does $x_r$. This completes the proof.
\end{proof}

\section{Jacobian Loci of the Invariants}\label{JacLoc}

\subsection{The General Linear Case}\label{JacLoc1}
If $f_1,...,f_l \in \K[\g_e^\ast]$ and $U \subseteq \g_e^\ast$ is a subspace then we denote by $$J_U(f_i : i = 1,...,l)$$ the \emph{Jacobian locus of the $f_i$ in $U$}: that is the set of $\gamma \in \g_e^\ast$ for which the differentials $d_\gamma f_1|_U,...,d_\gamma f_l|_U $ are linearly dependent. Define $$J := J_{\g_e^\ast}(x_r : r=1,...,N).$$ We aim to show that the condition on the codimension of the Jacobian locus in Skryabin's theorem (see \cite[Thereom~5.4]{Skr} or Theorem~\ref{skryabin} below) is satisfied for these invariants. We shall prove the following.
\begin{prop}\label{Codimension 2}
\emph{codim}$_{\g_e^\ast} J \geq 2$.
\end{prop}

In the style of \cite[$\S 3$]{PPY} we proceed by identifying a 2-dimensional plane in $\g_e^\ast$ intersecting $J$ only at 0. The proposition will then quickly follow. Calculating the differentials of our invariants polynomials explicitly is not feasible, and so we infer the necessary properties implicitly. Fix $\gamma \in \g_e^\ast$ and consider the polynomial
\begin{eqnarray*}
x_r^\gamma &:& \mathfrak{g}_e^\ast \ra \K\\
 x_r^\gamma& :&v \longmapsto x_r(\gamma + v).
\end{eqnarray*}
It is easily seen that
\begin{eqnarray}\label{xrdiff}
x_r^\gamma = \text{ constant } + d_{\gamma} x_r + \text{ higher degree polynomials in } \K[\g_e^\ast]
\end{eqnarray}

Hence in order to show that $\gamma \in \g_e^\ast \backslash J$ it will suffice to show that the linear components of $x_1^\gamma|_U,...,x_N^\gamma|_U$ are linearly independent for some appropriate choice of vector space $U \subseteq \g_e^\ast$.

Let us define
\begin{eqnarray}\label{alphabeta}
\alpha = \sum_{1 \leq i\leq n} a_i (\xi_i^{i,\lambda_i-1})^\ast \text{ and } \beta = \sum_{i=2}^{n} (\xi_{i-1}^{i,\lambda_{i}-1})^\ast
\end{eqnarray}
where the coefficients $a_i \in \K$ are subject to the constraint that $a_i = a_j$ implies $i=j$. The following lemma is essentially due to Brown and Brundan \cite[Lemma~4.2]{BB}. Their paper makes use of a basis for $\g_e$ which is defined combinatorially, and uses the reverse ordering for the Jordan block sizes. We translate their argument into the notation of our basis $\{ \xi_i^{j,s}\}$ for the reader's convenience.

\begin{lem}\label{beta}
For all $t \in \K^\times$ we have $t \beta \in \g_e^\ast \backslash J$.
\end{lem}
\begin{proof}
Since $J$ is a cone it shall suffice to prove the lemma in case $t=1$. Let $$U = \spn\{(\xi_i^{1,s})^\ast: 1 \leq i \leq n, \lambda_1 - \lambda_i \leq s < \lambda_1\}\subseteq \g_e^\ast$$ and allow
\begin{eqnarray*}
v = \sum_{i=1}^n \sum_{s=\lambda_1 - \lambda_i}^{\lambda_1 - 1} c_{i,s}( \xi_i^{1,s} )^\ast
\end{eqnarray*}
to be an arbitrary element of $U$, with $c_{i,s} \in \K$. We have
\begin{eqnarray}\label{xievaluate}
\xi_i^{j,s}(\beta + v) = \left\{ \begin{array}{ll}
         1 & \mbox{ if $i=j-1$ and $s = \lambda_j -1$}\\
        c_{i,s} & \mbox{ if $j = 1$}\\
        0 & \mbox{ otherwise} \end{array} \right.
\end{eqnarray}

Observe that if $1 \leq r \leq \lambda_1$ then $x_r = \sum_i \xi_i^{i, r-1}$ so that $x_r( \beta + v) = c_{1,r-1}$ and $x_r^{\beta}|_U= \xi_1^{1,r-1}$. If $n=1$ then $\lambda_1 = N$ and we have shown that the linear terms of the $x_1^{\beta}|_U, ..., x_N^{\beta}|_U$ are linearly independent and by formula (\ref{xrdiff}) we are done.

Assume not, so that we may choose $\lambda_1 < r \leq N$ and by the definition of the sequence of degrees we have $d := d_r > 1$. Let us assume that $(w,\mu) \in \gS_d \times\mathcal{C}_\lambda$ is a pair such that $$\Theta_r(w,\mu)(\beta + v) = (\sgn(w) \xi_{i_1}^{i_{w1}, s_1} \cdots \xi_{i_d}^{i_{wd}, s_d})(\beta + v) \neq 0.$$ We require that $\xi_{i_k}^{i_{wk}, s_k}(\beta + v) \neq 0$ for all $k=1,...,d$. 

For any $w \in \gS_d$ we can be sure that $wk \leq k$ for some $k \in \{1,...,d\}$. Fix such a $k$. Due to (\ref{xievaluate}) we have $i_{wk} = 1$. Clearly $wl \neq 1$ for $l \neq k$ so, again by (\ref{xievaluate}), $i_{wl} = i_l + 1$ for all $l \neq k$. This gives us $i_{wk} = 1$, $i_{w^2k} = 2$, $i_{w^3k} = 3,...$ and $i_{w^dk} = i_{k} = d$. Furthermore by (\ref{xievaluate}) we must have $\mu_1 = \lambda_1, \mu_2 = \lambda_2, ..., \mu_{d-1} = \lambda_{d-1}$. Since $|\mu | = r$ and $l(\mu) = d$ we have $\mu_d = r - \sum_{j=1}^{d-1} \lambda_j$. The upshot is that $w$ must be the $d$-cycle $(123\cdots d)$ and that $$\mu = (\lambda_1, \lambda_2,..., \lambda_{d-1}, r - \sum_{j=1}^{d-1} \lambda_j, 0, 0,..., 0).$$

We make the notation $t_r = r -\sum_{j=1}^{d-1} \lambda_j \in \{1,...,\lambda_d\}$. Since $\sgn(123\cdots d) = (-1)^{d-1}$ we have shown that $x_r(\beta +v) = (-1)^{d-1} c_{d, t_r-1}$. Thus $$x_r^{\beta}|_U = (-1)^{d-1} \xi_d^{1,t_r-1}|_U$$ for all $r=1,...,N$. These linear functions are clearly linearly independent and by formula (\ref{xrdiff}) the lemma is proven.
\end{proof}

For the next lemma we shall need to make use of a decomposition of $\g_e$ similar to the well known triangular decomposition of $\g$. We define 
\begin{eqnarray*}
\mathfrak{n}^- &:=& \spn\{\xi_i^{j,s} : i < j\} \\
\smallskip
\mathfrak{h} &:=& \spn\{\xi_i^{j,s} : i = j\} \\
\smallskip
\mathfrak{n}^+ &:=& \spn\{\xi_i^{j,s} : i > j\}
\end{eqnarray*}
where $\lambda_j - \min(\lambda_i, \lambda_j) \leq s < \lambda_j$ in all three of the above. If we order the basis $\{e^s w_i\}$ so that $e$ is in Jordan normal form then $\mathfrak{n}^-$ is strictly lower triangular and $\mathfrak{n}^+$ is strictly upper triangular, and $\g_e = \n^- \oplus \hh \oplus \n^+$. Of course $\hh$ is not actually a torus unless $e = 0$, however it is proven over $\C$ in \cite[$\S 5$]{Yak1} that $\hh$ is a generic stabaliser in $\g_e$, and that $(\g_e)_\alpha = \hh$. We shall see later that this generic stabaliser also exists when we work over $\K$ (Theorem~\ref{genstabal}). We remark the definition of $\alpha$ actually originates in \cite{Yak1}.

There is a dual decomposition $$\g_e^\ast = (\mathfrak{n}^-)^\ast \oplus \mathfrak{h}^\ast \oplus (\mathfrak{n}^+)^\ast$$ where $\hh^\ast$ is defined to be the annihilator of $\n^-\oplus \n^+$ in $\g_e^\ast$, and similar for $(\n^-)^\ast$ and $(\n^+)^\ast$. We have $\alpha \in \mathfrak{h}^\ast$ and $\beta \in (\mathfrak{n}^-)^\ast$.

\begin{lem}\label{alpha}
There exists $g \in G_e$ such that \emph{Ad}$^\ast(g) \alpha = \alpha + \beta$.
\end{lem}
\begin{proof}
Since each $v \in \mathfrak{n}^+$ is nilpotent, the translation morphism $v \ra 1 + v$ takes $v \in \mathfrak{n}^+$ to a unipotent matrix in $G_e$. We denote the subgroup generated by all $1 + v$ with $v \in \mathfrak{n^+}$ by $N^+$. It is easily checked that $1 + \mathfrak{n}^+$ is closed under matrix multiplication. Due to formula~(4) in the proof of Corollary~\ref{invariants} the set $1 + \mathfrak{n}^+$ is also closed under the map $g \mapsto g^{-1}$. Hence $N^+ = 1 + \mathfrak{n}^+$. We aim to prove that $\Ad^\ast(N^+) \alpha = \alpha + (\mathfrak{n}^-)^\ast$, from which our proposition will quickly follow. The one dimensional subspaces $\{1 + t \xi_i^{j,s} : t \in \K\}$ with $i > j$ generate $N^+$. Since $i > j$ we see that $(1 + t\xi_i^{j,s})^{-1} = 1 - t \xi_i^{j,s}$ (this is observed in the proof of Corollary~\ref{invariants}). A quick calculation then shows that
\begin{eqnarray*}
\Ad^\ast(1 + t \xi_i^{j,s}) \alpha = \alpha + t(a_j(\xi_j^{i, \lambda_j - 1- s})^\ast - a_i(\xi_j^{i, \lambda_i - 1- s})^\ast) \in \alpha + (\mathfrak{n}^-)^\ast
\end{eqnarray*}
The conditions on the $a_i$ ensure that the linear forms $\{a_j(\xi_j^{i, \lambda_j - 1- s})^\ast - a_i(\xi_j^{i, \lambda_i - 1- s})^\ast : i > j \}$ are linearly independent, hence span $(\mathfrak{n}^-)^\ast$. We see that $\dim(\Ad^\ast(N^+) \alpha) = \dim(\alpha + (\mathfrak{n}^-)^\ast$). Thanks to \cite[Theorem~2]{Ros} we know that $\Ad^\ast(N^+) \alpha$ is a closed subvariety of $\alpha + (\mathfrak{n}^-)^\ast$. The dimensions coincide and so we have equality $\Ad^\ast(N^+) \alpha = \alpha + (\mathfrak{n}^-)^\ast$. Now $\beta \in (\mathfrak{n}^-)^\ast$ so there exists some $g \in N^+$ such that $\Ad^\ast(g) \alpha = \alpha + \beta$ as required.
\end{proof}

Let $a : \K^\times \ra G_e$ be the cocharacter given by $a(t)w_i = t^i w_i$. Define a rational linear action $\rho : \K^\times \ra GL(\g_e^\ast)$ by
\begin{eqnarray*}
\rho(t) \gamma = t \Ad^\ast (a(t)) \gamma
\end{eqnarray*}
where $\gamma \in \g_e^\ast$ and $t \in \K^\times$. Clearly we have $\rho(t) (\xi_i^{j,s})^\ast = t^{i-j+ 1} (\xi_i^{j,s})^\ast$.

\begin{lem}\label{Jstable}
The Jacobian locus $J$ is
\begin{enumerate}
\item{$G_e$-stable;}
\smallskip
\item{$\rho(\K^\times)$-stable.}
\end{enumerate}
\end{lem}

\begin{proof}
Since $x_r$ is $G_e$-invariant
\begin{eqnarray*}
x_r(\Ad^\ast (g)(\gamma + \delta)) = x_r (\gamma + \delta)
\end{eqnarray*}
for all $g \in G_e$ and $\gamma,\delta \in \g_e^\ast$. This equates to $x_r^{\Ad^\ast(g)\gamma} \circ (\Ad^\ast g) = x_r^\gamma$. The linear part of the left hand side of this equation is $d_{\Ad^\ast(g)\gamma} x_r \circ (\Ad^*g)$ and the same of the right hand side is $d_\gamma x_r$. As $\Ad^\ast (g)$ is invertible, the dimension of the linear span of $d_{\Ad^\ast(g)\gamma} x_1,..., d_{\Ad^\ast(g)\gamma} x_N$ equals that of the $d_{\gamma} x_1,..., d_{\gamma} x_N$, whence 1.

Turning our attention to $\rho(\K^\times)$, fix $t \in \K^\times$, $r = 1,...,N$, $(w,\mu) \in \gS_d \times \mathcal{C}_\lambda$ and observe that
\begin{eqnarray*}
\Theta_r(w,\mu) \circ \rho(t) &=& (\sgn(w) \xi_{i_1}^{i_{w1}, s_1} \cdots \xi_{i_d}^{i_{wd}, s_d})\circ \rho(t) \\
&=& (\prod_{k=1}^d t^{i_k - i_{wk} +1})\sgn(w)\xi_{i_1}^{i_{w1}, s_1} \cdots \xi_{i_d}^{i_{wd}, s_d}\\
& =& t^{d}(\sgn(w) \xi_{i_1}^{i_{w1}, s_1} \cdots \xi_{i_d}^{i_{wd}, s_d}) = t^d \Theta_r(w, \mu).
\end{eqnarray*}
So that $x_r \circ \rho(t) = t^{d} x_r$. Next let $\gamma, v \in \g_e^\ast$ and observe that $$x_r(\rho(t)\gamma + v) = x_r\circ \rho(t) (\gamma + \rho(t)^{-1} v) = t^d x_r(\gamma + \rho(t)^{-1}v)$$ which is written as $x_r^{\rho(t)\gamma} = t^d x_r^\gamma\circ \rho(t)^{-1}$ in our notations. We conclude that the linear terms must coincide, so that $$d_{\rho(t)\gamma} x_r = t^d d_{\gamma} x_r \circ \rho(t)^{-1}.$$ However, $\rho(t)^{-1}$ is evidently invertible so 2 follows.
\end{proof}

\begin{lem}\label{plane1}
$(\K \alpha \oplus \K\beta) \cap J = 0$.
\end{lem}

\begin{proof}
Let $t_1, t_2 \in \K$ and $\gamma = t_1 \alpha + t_2 \beta \neq 0$. We shall show that $\gamma \in \g_e^\ast \backslash J$. If $t_2 = 0$ then the element $g \in G_e$ constructed in Lemma~\ref{alpha} sends $\gamma$ to $t_1\alpha + t_1\beta$ so by part 1 of Lemma~\ref{Jstable} it suffices to prove that $\gamma \in \g_e^\ast \backslash J$ whenever $t_2 \neq 0$. By Lemma~\ref{beta} we may also assume that $t_1 \neq 0$.

It is clear that $\rho(t)\alpha = t \alpha$ and $\rho(t) \beta = \beta$. Consider the variety $\K \alpha + t_2\beta$. Since $J$ is closed it follows that $(\g_e^\ast \backslash J) \cap (\K \alpha + t_2\beta)$ is a Zariski open subset of $\K \alpha + t_2\beta$. By Lemma~\ref{beta} that intersection is non-empty. We deduce that $$\rho(\K^\times)\gamma \cap (\g_e^\ast \backslash J) = (\K^\times\alpha + t_2\beta) \cap (\g_e^\ast \backslash J)  \neq \emptyset.$$ By part 2 of Lemma~\ref{Jstable}, $\g_e^\ast \backslash J$ is $\rho(\K^\times)$-stable implying $\gamma \in \g_e^\ast \backslash J$ as required.
\end{proof}

We are now ready to prove Proposition~\ref{Codimension 2}.
\begin{proof}
The Jacobian locus is conical and Zariski closed. Apply the above lemma.
\end{proof}

\subsection{The Symplectic and Orthogonal Cases}\label{JacLoc2}

In this section we aim to prove an analogue of Proposition~\ref{Codimension 2} for centralisers in other classical types.
\begin{prop}\label{jaclocsymorth}The following are true:
\begin{enumerate}
\item{Let $\epsilon = -1$ so that $K$ is of type $C$. Then \emph{codim}$_{\h_e^\ast} J_{\h_e^\ast}(x_{r}: r $ even$) \geq 2$;}
\smallskip
\item{Let $\epsilon = 1$ so that $K$ is of type $B$ or $D$. Then \emph{codim}$_{\pp_e^\ast} J_{\pp_e^\ast}(x_{r}: r + d_r $ even$) \geq 2$.}
\end{enumerate}
\end{prop}
As was discussed in the introduction, the proofs of parts 1 and 2 shall be identical. We shall supply all details of the proof of part 1, whilst our proof of part 2 shall simply consist of a description of necessary changes to that proof. As an immediate corollary to Proposition~\ref{jaclocsymorth} we obtain.
\begin{cor}\label{nonzero1} The following are true:
\begin{enumerate} 
\item{If $\epsilon = -1$ then the restrictions $x_r|_{\h_e^\ast}$ with $r$ even are non-zero and distinct;}
\smallskip
\item{If $\epsilon = 1$ then the restrictions $x_r|_{\pp_e^\ast}$ with $r+d_r$ even are non-zero and distinct.}
\end{enumerate}
\end{cor}
\begin{proof}
If $x_r|_{\h_e^\ast} = x_s|_{\h_e^\ast}$ for some $r \neq s$ then we would have $\h_e^\ast = J_{\h_e^\ast}(x_r : r \text{ even})$, contradicting part 1 of Proposition~\ref{jaclocsymorth}. Similarly the restrictions are non-zero. The same argument applies for part 2.
\end{proof}
\emph{Until stated otherwise we assume $\epsilon = -1$ so that $K$ is of type $C$}. Since $\g_e = \h_e \oplus \pp_e$ there is a natural inclusion $\iota : S(\h_e) \ra S(\g_e)$ which is also a $\K$-algebra homomorphism.
\begin{lem}\label{dees}
$d_\gamma x_r = d_\gamma \iota(x_r|_{\h_e^\ast})$ for $r$ even and for all $\gamma \in \h_e^\ast$.
\end{lem}
\begin{proof}
Let $\gamma \in \h_e^\ast$. Let $\BB_0$ and $\BB_1$ be bases for $\h_e$ and $\pp_e$ respectively, so that $\BB := \BB_0 \cup \BB_1$ is a basis for $\g_e$. We aim to show that $d_\gamma (x_r - \iota(x_r|_{\h_e^\ast})) = 0$. Now $x_r - \iota(x_r|_{\h_e^\ast})$ may be written as a finite sum $\sum_{i=1}^k c_i m_i$ where $c_i \in \K^\times$ are constants and  the $m_i \in S(\g_e)$ are monomials in the basis $\BB$. Since $\iota(x_r|_{\h_e^\ast})$ is the sum of those monomial summands of $x_r$ which contain no factors from $\BB_1$ we conclude that each $m_i$ possesses a factor from $\BB_1$. We have $\sigma \eta = -\eta$ for each $\eta \in \BB_1$ and, since $r$ is even, $\sigma x_r = x_r$ by Proposition~\ref{sigmaxr} so there must be an even number of factors from $\BB_1$ in each monomial summand of $x_r$. This implies that each $m_i$ possesses at least two factors from $\BB_1$ and that for all $x \in \BB$ the partial derivative $\frac{\partial m_i}{\partial x}$ either is zero or possesses at least one nonzero factor from $\BB_1$. The functionals $\BB_1$ annihilate $\h_e^\ast$ so $$\frac{\partial m_i}{\partial x}(\gamma) = 0$$ for all $x \in \BB$. But now $$d_\gamma (x_r - \iota(x_r|_{\h_e^\ast})) = \sum_{x \in \BB} \frac{\partial (x_r - \iota(x_r|_{\h_e^\ast}))}{\partial x}(\gamma)x = \sum_{x \in \BB} \sum_{i=1}^k c_i \frac{\partial m_i}{\partial x}(\gamma)x = 0.$$ The lemma follows.
\end{proof}
\begin{lem}\label{jal}
$J_{\h_e^\ast}(x_r : r \text{ even}) = \h_e^\ast \cap J_{\g_e^\ast}(x_r : r \text{ even})$
\end{lem}
\begin{proof}
Fix $\gamma \in \h_e^\ast$. If $\sum_r c_{2r} d_\gamma x_{2r} = 0$ then $$(\sum_r c_{2r} d_\gamma x_{2r})|_{\h_e^\ast} = \sum_r c_{2r} d_\gamma (x_{2r}|_{\h_e^\ast}) = 0$$ which gives one inclusion. Conversely suppose $\sum_r c_{2r} d_\gamma (x_{2r}|_{\h_e^\ast}) = 0$. Then $$\iota(\sum_r c_{2r} d_\gamma (x_{2r}|_{\h_e^\ast})) = \sum_r c_{2r} d_\gamma \iota(x_{2r}|_{\h_e^\ast}) =  \sum_r c_{2r} d_\gamma x_{2r} = 0,$$ which gives the other inclusion.
\end{proof}
Let $\alpha$ be as defined in the previous section, with the additional constraint that $a_i = - a_{i'}$ for all $i \neq i'$. We define $$\bar\beta = \beta + \beta' \text{ where } \beta' = \sum_{i+1 \neq i'} \varepsilon_{i,i+1,0} (\xi_{(i+1)'}^{i', \lambda_i-1})^\ast.$$
\begin{rem}\label{remremree}
\rm{These definitions for $\alpha$ and $\bar{\beta}$ are rather unclear at first glance. They first appeared in \cite{Yak1}, were used again in \cite{PPY}, and have a simple rationale behind them which we shall briefly discuss. The obvious guess of how to construct analogues of $\alpha$ and $\beta$, but lying in $\h_e^\ast$, is to define an automorphism $\sigma^\ast$ which acts by $+1$ on $\h_e^\ast$ and acts by $-1$ on $\pp_e^\ast$, and extends to all of $\g_e^\ast$ by linearity. Just as we obtained a spanning set for $\h_e$ by considering expressions $x + \sigma(x)$ with $x \in \g_e$ we may obtain analogues for $\alpha$ and $\beta$ by considering $\alpha + \sigma^\ast \alpha$ and $\beta + \sigma^\ast \beta$. This is roughly what we do here, although some of the summands $(\xi_i^{j,s})^\ast$ are rescaled in order to simplify notation.}
\end{rem}
In Section~\ref{elemInv} we introduced dual vectors $(\zeta_i^{j,s})^\ast := (\xi_i^{j,\lambda_j-1-s})^\ast + \varepsilon_{i,j,s} (\xi_{j'}^{i',\lambda_i - 1 - s})^\ast$. In order to carry out explicit calculations using $\alpha$ and $\bar\beta$ it will be necessary to express these two elements in terms of the $(\zeta_i^{j,s})^\ast$.
\begin{lem}\label{albeso} If $a_i = - a_{i'}$ for $i \neq i'$ then
\begin{eqnarray*}
\alpha = \frac{1}{2} \sum_{i = i'} a_i(\zeta_i^{i,0})^\ast +  \sum_{i < i'}a_i(\zeta_i^{i,0})^\ast; \\
\bar\beta = \frac{1}{2}\sum_{i+1 = i'}(\zeta_i^{i+1, 0})^\ast + \sum_{i+1 \neq i'}(\zeta_i^{i+1,0})^\ast;
\end{eqnarray*}
and in particular $\alpha, \bar\beta \in \h_e^\ast$.
\end{lem}
\begin{proof}
Since $\epsilon = -1$, Lemma~\ref{nilpotents} implies that $i = i'$ if and only if $\lambda_i$ is even. Using Lemma~\ref{spanningdetails} it follows that $(\zeta_i^{i,0})^\ast = 2(\xi_i^{i,\lambda_i-1})^\ast$ for all $i = i'$ and $(\zeta_i^{i,0})^\ast = (\xi_i^{i,\lambda_i-1})^\ast -(\xi_{i'}^{i',\lambda_i-1})^\ast$ for all $i < i'$. The formula for $\alpha$ follows. Similarly $\varepsilon_{i,i+1,0} = (-1)^{\lambda_{i+1}}\varpi_{i\leq i'} \varpi_{i+1\leq (i+1)'}$. If $i+1 = i'$ then Lemma~\ref{nilpotents} implies $\lambda_{i+1}$ is odd and that $\varepsilon_{i,i+1,0} = 1$. We conclude that $(\zeta_i^{i+1,0})^\ast = 2(\xi_i^{i+1, \lambda_{i+1} - 1})^\ast$ which completes the proof.
\end{proof}
Recall that $J := J_{\g_e^\ast}(x_1,...,x_N)$ and that there is a rational linear action $\rho : \K^\times \ra GL(\g_e^\ast)$ defined preceding Lemma~\ref{xrdiff}. If $i + 1 \neq i'$ then $(i+1)' > i'$ and so $$\rho(t) (\xi_{(i+1)'}^{i', \lambda_i - 1})^\ast = t^{k_i}  (\xi_{(i+1)'}^{i', \lambda_i - 1})^\ast$$ where $k_i = (i+1)' - i' + 1 \geq 2$.

\begin{lem}\label{plane2}
$(\K \alpha \oplus \K\bar\beta) \cap J = 0$
\end{lem}
\begin{proof}
Let $\gamma = t_1 \alpha + t_2\bar\beta \neq 0$ with $t_1,t_2 \in \K$. We shall show that $\gamma \in \g_e^\ast \backslash J$. Supposing $t_1 \neq 0$ and $t_2 = 0$ we may invoke Lemma~\ref{plane1} to conclude that $\gamma \in \g_e^\ast \backslash J$.

Suppose $t_2\neq 0$, then consider the set $E = \{t\alpha + (\beta + \sum \varepsilon_{i,i+1, 0} t^{k_i}\xi_{(i+1)'}^{i',\lambda_i-1}) : t \in \K\}$ where $k_i \in \N$ is defined preceding the statement of the lemma. It is a one dimenisonal variety containing $\beta$, hence by Lemma~\ref{beta} it must intersect the set $\g_e^\ast \backslash J$ in a non-empty open subset. Since $\rho(\K^\times)\gamma = \{t\alpha + (\beta + \sum \varepsilon_{i,i+1, 0} t^{k_i}\xi_{(i+1)'}^{i',\lambda_i-1}) : t \in \K^\times\} \subseteq E$ is also non-empty and open in $E$ the intersection $\rho(\K^\times)\gamma \cap (\g_e^\ast \backslash J)$ is non-empty. By part 2 of Lemma~\ref{Jstable}, $J$ is $\rho(\K^\times)$-stable and so $\gamma \in \g_e^\ast \backslash J$.
\end{proof}

We can now give a proof for part 1 of Proposition~\ref{jaclocsymorth}.
\begin{proof}
By Lemmas~\ref{albeso} and \ref{plane2} there is a 2 dimensional plane contained in $\h_e^\ast$ intersecting $J$ only at zero. By Lemma~\ref{jal} we have $J_{\h_e^\ast}(x_r : r \text{ even}) = \h_e^\ast \cap J_{\g_e^\ast}(x_r : r \text{ even}) \subseteq \h_e^\ast \cap J$ so that same plane intersects $J_{\h_e^\ast}(x_r : r \text{ even})$ only at zero. As $J_{\h_e^\ast}(x_r : r \text{ even})$ is conical and Zariski closed, the proposition follows.
\end{proof}

In order to prove part 2 of Proposition~\ref{jaclocsymorth} we follow exactly the same scheme of argument as above. Since we treated the symplectic case so carefully, our proof here will constantly refer back to previous arguments. \emph{For the remnant of this subsection take $\epsilon = 1$}. Again we have an inclusion $\iota : S(\pp_e) \ra S(\g_e)$. The next lemma is analogous to Lemma~\ref{jal}.
\begin{lem}\label{jal2}
$J_{\pp_e^\ast}(x_r : r + d_r \text{ even}) = \pp_e^\ast \cap J_{\g_e^\ast}(x_r : r + d_r \text{ even})$
\end{lem}
\begin{proof}
First we prove a version of Lemma~\ref{dees}: that $d_\gamma x_r = d_\gamma i(x_r|_{\pp_e^\ast})$ for all $\gamma \in \pp_e^\ast$. Resume notations $\BB_0$, $\BB_1$ and $\BB$ from Lemma~\ref{dees}. This time write $x_r - \iota(x_r|_{\pp_e^\ast}) = \sum_{i=1}^k c_i m_i$ where $c_i \in \K^\times$ are non-zero constants and $m_i$ are monomials in $\BB$. Since $\iota(x_r|_{\pp_e^\ast})$ is just the sum of those monomials summands of $x_r$ which contain no terms from $\BB_0$, so each $m_i$ possesses a factor from $\BB_0$. Using a reasoning identical to Lemma~\ref{vanishing1} we see that the number of such factors is even. The proof now concludes exactly as per Lemma~\ref{dees}. In order to finish the current proof we use identical calculations to Lemma~\ref{jal}, simply replacing the set $\{x_r : r \text{ even}\}$ with $\{x_r: r + d_r \text{ even}\}$, and restricting our functions to $\pp_e^\ast$ rather than $\h_e^\ast$.
\end{proof}
Next we identify a 2-dimensional plane contained in $\pp_e^\ast$ intersecting $J_{\pp_e^\ast} (x_r : r + d_r \text{ even})$ only at zero. As was noted in Remark~\ref{remremree} our construction in type $C$ is essentially to take $\alpha + \sigma^\ast \alpha$ and $\beta + \sigma^\ast \beta$. The obvious choice when constructing elements in $\pp_e^\ast$ is to consider $\alpha - \sigma^\ast \alpha$ and $\beta - \sigma^\ast \beta$. This is essentially what we do here. 

Define $\alpha$ in the same way as in Section~\ref{JacLoc1}, with $a_i = -a_{i'}$ for all $i \neq i'$, and define $\bar\beta = \beta - \beta'$.
\begin{lem}\label{albeos} We have
\begin{eqnarray*}
\alpha = \frac{1}{2} \sum_{i = i'} a_i(\eta_i^{i,0})^\ast +  \sum_{i < i'}a_i(\eta_i^{i,0})^\ast;\\
\bar\beta = \frac{1}{2}\sum_{i+1 = i'}(\eta_i^{i+1, 0})^\ast + \sum_{i+1 \neq i'}(\eta_i^{i+1,0})^\ast;
\end{eqnarray*}
and in particular $\alpha, \bar\beta \in \pp_e^\ast$.
\end{lem}
\begin{proof}
Simply follow the proof of Lemma~\ref{albeso} verbatim, replacing each occurrence of $\zeta_i^{j,s}$ with $\eta_i^{j,s}$, and exchanging the words odd and even.
\end{proof}
We can now give a version of Lemma~\ref{plane2}. The statement of the lemma is precisely the same, but we remind the reader that now we have $\epsilon = 1$, and our definition of $\bar\beta$ is slightly different.
\begin{lem}
$\K \alpha \oplus \K \bar\beta \cap J = 0$
\end{lem}
\begin{proof}
The argument is identical to Lemma~\ref{plane2} except in this instance the correct definition of $E$ is  $\{t\alpha + (\beta - \sum \varepsilon_{i,i+1, 0} t^{k_i}\xi_{(i+1)'}^{i',\lambda_i-1}) : t \in \K\}$, which reflects the fact that $\bar\beta = \beta - \beta'$.
\end{proof}
We may now supply the proof of part 2 of Proposition~\ref{jaclocsymorth}. 
\begin{proof}
By Lemmas~\ref{plane2} and \ref{albeos} there is a two dimensional plane contained in $\pp_e^\ast$ intersecting $J$ only at zero. By lemma \ref{jal2} we have $J_{\pp_e^\ast}(x_r : r + d_r \text{ even}) = \pp_e^\ast \cap J_{\g_e^\ast}(x_r : r + d_r \text{ even}) \subseteq \pp_e^\ast \cap J$ so that same plane intersects $J_{\pp_e^\ast}(x_r : r + d_r \text{ even})$ only at zero. As $J_{\pp_e^\ast}(x_r : r + d_r \text{ even})$ is conical and Zariski closed, the proposition follows.
\end{proof}

\section{Generic Stabalisers}\label{GenEl}

Here we make a quick detour to discuss the existence of generic stabalisers in certain cases. We define generic stabalisers in a general setting. Suppose $\tG$ is an algebraic group with identity element $1$ and $\tg = \Lie(\tG)$. If $W$ is a $\K$-vector space and $\rho : \tG \ra GL(W)$ is a rational representation then $d_1\rho : \tg \ra \gl(W)$ is a representation of $\tg$. We say that $w \in W$ is a generic point, and that $\tg_w$ is a generic stabaliser if there exists a Zariski open subset $\mathcal{O} \subseteq W$ such that for all $v \in \mathcal{O}$, the centralisers $\tg_w$ and $\tg_v$ are $\tG$-conjugate under the adjoint action of $\tG$ on $\tg$.

When completing the proofs of Theorems~\ref{main1} and \ref{main3} we shall require to know the index of $\g_e$, of $\h_e$ when $\epsilon = -1$, and of $\h_e$ in $\pp_e$ when $\epsilon = 1$. The first two of these indexes are actually already known and calculating $\ind(\h_e,\pp_e)$ when $\epsilon = 1$ is the main purpose of this section. This shall be achieved in Corollary~\ref{orthindex}. In order to calculate the index we prove the existence of a generic stabaliser.
The existence of a generic stablisers of $\g_e$ in $\g_e^\ast$ and of $\h_e$ in $\h_e^\ast$ when $\epsilon = -1$ was proven in \cite[$\S 5$]{Yak1} over $\C$ using a powerful criterion due to Elashvili \cite[$\S 1$]{Ela}, which is particular to characteristic zero. We extract the scheme of argument from that criterion and with a little extra work our results follow over $\K$.

The following lemma is very well known and embodies a common line of reasoning for studying generic stabalisers. We include the proof for the reader's convenience. The proof and, in fact, all results contained in this section are characteristic free. As such we take $\tG \subseteq GL(W)$ to be an algebraic group over $\K$ or $\C$ with Lie algebra $\tg$.
\begin{lem}\label{genreg}
Fix $\gamma \in \tg$ and let $\V_\gamma = \{\delta \in W : \tg_\gamma \subseteq \tg_{\delta}\}$. If \begin{eqnarray*}\varphi: \tG \times \V_\gamma \ra W;\\ \varphi(g,w) = g\cdot w\end{eqnarray*} is a dominant morphism then $\gamma$ is a regular, generic point for the action of $\tG$ on $W$.
\end{lem}
\begin{proof}
Suppose $\mathcal{P}$ is an open subset of $W$ contained in $\varphi(\tG \times \V_\gamma)$. It is well known that there exists an open subset $\mathcal{O}$ of $W$ such that the stabalisers of all points in $\mathcal{O}$ have dimension $\ind(\tg, W)$ (argue in the style of \cite[1.11.5]{Dix}). Then there exists $\delta \in \mathcal{O} \cap \mathcal{P}$, ie. there is $(g, \delta') \in \tG \times \V_\gamma$ such that $\varphi(g,\delta')= \delta$ and $\dim(\tg_\delta) = \ind(\tg, W)$. We conclude that $$\dim(\tg_\gamma) \leq \dim(\tg_{\delta'}) = \dim(\tg_\delta) = \ind(\tg, W)$$ so that $\gamma$ is regular. For all $\gamma' \in \mathcal{O} \cap \mathcal{P}$ we know that $\tg_{\gamma'}$ contains some $\tG$-conjugate of $\tg_\gamma$, say $\Ad(g) \tg_\gamma \subseteq \tg_{\gamma'}$. However by dimension considerations we see that $\tg_{\gamma'} = \Ad(g) \tg_\gamma$. Therefore $\mathcal{O} \cap \mathcal{P}$ is an open set within which the $\tg$-stabalisers of all points are $\tG$-conjugate.
\end{proof}
The following is the main theorem of this section. 
\begin{thm}\label{genstabal} The linear form $\alpha$ is
\begin{enumerate}
\item{a generic, regular point for the action of $G_e$ on $\g_e^\ast$;}
\smallskip
\item{a generic, regular point for the action of $K_e$ on $\h_e^\ast$ when $\epsilon = -1$;}
\smallskip
\item{a generic, regular point for the action of $K_e$ on $\pp_e^\ast$ when $\epsilon = 1$.}
\end{enumerate}
\end{thm}

The method is to satisfy the assumptions of Lemma~\ref{genreg}. As was the case in Section~\ref{JacLoc}, the arguments in the symplectic and orthogonal cases are almost identical. As such we shall prove parts 1 and 2 quite explicitly, and the describe how those arguments may be adapted to prove part 3.

Recall the dual decompositions $\g_e = \mathfrak{n}^- \oplus \mathfrak{h} \oplus \mathfrak{n}^+$ and $\g_e^\ast = (\mathfrak{n}^-)^\ast \oplus \mathfrak{h}^\ast \oplus (\mathfrak{n}^+)^\ast$ which were defined in Section~\ref{JacLoc1}, and recall also that $\alpha \in \mathfrak{h}^\ast$. Yakimova proves in \cite[Theorem~1]{Yak1} that $(\g_e)_\alpha = \hh$.

\begin{prop}\label{shost}
The following map is dominant
\begin{eqnarray*}
\vartheta : & G_e \times \mathfrak{h}^\ast \ra \g_e^\ast \\
\vartheta : & (g, \gamma) \mapsto \Ad^\ast(g) \gamma. 
\end{eqnarray*}
\end{prop}
\begin{proof}
By \cite[Theorem 3.2.20(i)]{Spr} it suffices to show that the differential $d_{(1, \alpha)} \vartheta : \g_e \oplus \hh^\ast \ra \g_e^\ast$ is surjective. In \cite[Lemma~1.6]{TY} the differential is calculated $$d_{(1,\alpha)} \vartheta (x, v) = \ad^\ast(x) \alpha + v.$$ Therefore $\hh^\ast \subseteq d_{(1,\alpha)} \vartheta (\g_e, \hh^\ast)$. To complete the proof we shall show that $(\mathfrak{n}^-)^\ast, (\mathfrak{n}^+)^\ast \subseteq \ad^\ast(\g_e)\alpha =d_{(1,\alpha)} \vartheta (\g_e, 0)$. A quick calculation shall confirm that 
\begin{eqnarray}\label{adform}
\ad^\ast (\xi_i^{j,s}) (\xi_k^{l,r})^\ast = \delta_{ik} (\xi_j^{l,r-s})^\ast - \delta_{jl} (\xi_k^{i,r-s})^\ast
\end{eqnarray}
and that
\begin{eqnarray}\label{eqnn}
\ad^\ast(\xi_i^{j,s}) \alpha = a_i (\xi_j^{i,\lambda_i-1-s})^\ast - a_j (\xi_j^{i,\lambda_j-1-s})^\ast.
\end{eqnarray} Fix $i > j$. If $\lambda_i = \lambda_j$ then the restriction $a_i \neq a_j$ implies that $(\xi_j^{i,s})^\ast \in \ad^\ast(\mathfrak{n}^+)\alpha$ for $s = 0,1,...,\lambda_i$. If $\lambda_i < \lambda_j$ then substituting $s=\lambda_j-1$ into equation (\ref{eqnn}) gives $(\xi_j^{i,0})^\ast \in \ad^\ast(\mathfrak{n}^+)\alpha$. Substituting successively smaller values of $s$ into equation (\ref{eqnn}) we obtain by induction $(\xi_j^{i,s})^\ast \in \ad^\ast(\mathfrak{n}^+)\alpha$ for all $s = 0,1,..., \lambda_i -1$. We conclude that $(\mathfrak{n}^-)^\ast \subseteq \ad^\ast(\mathfrak{n}^+) \alpha$. An identical argument shows that $(\mathfrak{n}^+)^\ast \subseteq \ad^\ast(\mathfrak{n}^-) \alpha$, completing the proof.
\end{proof}
In order to prove analogues of the above proposition for classical subalgebras of $\g$ we must first place further restriction on the coefficients $a_i$ which appear in the definition of $\alpha$. Define $$l_i = \left\{ \begin{array}{ll}
         1 & \mbox{if $i=i'$};\\
        2 & \mbox{if $i \neq i'$}.\end{array} \right.$$
and impose the restriction $l_i a_i = l_j a_j$ only if $i = j$.

\begin{prop}\label{alphaKegeneric}
The following map is dominant when $\epsilon = -1$
\begin{eqnarray*}
\vartheta : & K_e \times (\mathfrak{h}^\ast\cap \h_e^\ast) \ra \h_e^\ast \\
\vartheta : & (g, \gamma) \mapsto \Ad^\ast(g) \gamma. 
\end{eqnarray*}
\end{prop}
\begin{proof}
The proof is very similar to that of Proposition~\ref{shost}. Once again we may prove that the differential $d_{(1, \alpha)} \vartheta$ is surjective. The differential is $d_{(1,\alpha)} \vartheta (x,v) = \ad^\ast(x) \gamma + v$. Since $d_{(1,\alpha)} \vartheta (0,\hh^\ast\cap \h_e^\ast) = \mathfrak{h}^\ast \cap \h_e^\ast = \spn\{ (\zeta_i^{i,s})^\ast\}$ we may complete the proof by showing that $\spn\{(\zeta_i^{j,s})^\ast : i \neq j\}\subseteq \ad^\ast(\h_e) \alpha$. Once again we shall need explicit calculations. From formula (\ref{adform}) it follows that
\begin{eqnarray*}
\ad^\ast (\zeta_i^{j,s}) (\zeta_k^{l,r})^\ast & = & \delta_{ik} (\zeta_j^{l, \lambda_j - 1 + r-s})^\ast - \delta_{jl} (\zeta_k^{i, \lambda_i - 1 +r-s})^\ast \\
& + & \delta_{il'} \varepsilon_{klr}(\zeta_j^{k', \lambda_j - 1+r-s})^\ast - \delta_{jk'} \varepsilon_{klr} (\zeta_{l'}^{i,\lambda_i - 1 +r-s})^\ast 
\end{eqnarray*}
Recall that there is an expression for $\alpha$ in terms of $(\zeta_i^{i,0})^\ast$ derived in Lemma~ \ref{albeso} $$\alpha = \frac{1}{2} \sum_{i = i'} a_i(\zeta_i^{i,0})^\ast +  \sum_{i < i'}a_i(\zeta_i^{i,0})^\ast = \sum_{i \leq i'}a_i (1 - \frac{1}{2} \delta_{i,i'}) (\zeta_i^{i,0})^\ast.$$
By Lemma~\ref{spanningdetails} we have $\varepsilon_{i,i,0} = (-1)^{\lambda_i}$ and as a consequence we have
\begin{eqnarray*}
\ad^\ast(\zeta_i^{j,s}) \alpha & = & (1 - \frac{1}{2}\delta_{i,i'}) a_i  (\zeta_j^{i,\lambda_j - 1 - s})^\ast - (1 - \frac{1}{2}\delta_{j,j'})a_j (\zeta_j^{i,\lambda_i - 1 - s})^\ast\\
& - & \varepsilon_{i',i',0} (1 - \frac{1}{2}\delta_{i,i'}) a_{i'}  (\zeta_j^{i,\lambda_j - 1 - s})^\ast - \varepsilon_{j',j',0}(1 - \frac{1}{2}\delta_{j,j'})a_j (\zeta_j^{i,\lambda_i - 1 - s})^\ast\\
& = & (1 - \frac{1}{2}\delta_{i,i'}) (a_i + (-1)^{\lambda_i}a_{i'})  (\zeta_j^{i,\lambda_j - 1 - s})^\ast - (1 - \frac{1}{2}\delta_{j,j'})(a_j + (-1)^{\lambda_j}a_{j'}) (\zeta_j^{i,\lambda_i - 1 - s})^\ast
\end{eqnarray*}
By Lemma~\ref{nilpotents}, $\lambda_i$ is even if and only if $i = i'$. Furthermore $a_i = a_{i'}$ whenever $i=i'$, and $a_i = -a_{i'}$ whenever $i\neq i'$. In either case $a_i + (-1)^{\lambda_i} a_{i'} = 2a_i$. We deduce that $$\ad^\ast(\zeta_i^{j,s}) \alpha = l_i a_i (\zeta_j^{i,\lambda_j-1-s})^\ast - l_j a_j (\zeta_j^{i,\lambda_i-1-s})^\ast.$$ Fix $i\neq j$. Suppose $\lambda_i = \lambda_j$. Then the restrictions imposed on the coefficients $a_i$ and $l_i$ ensure that $(\zeta_j^{i,s})^\ast \in \ad^\ast(\h_e)\alpha$ for $s=0,1,...,\lambda_i-1$. Now suppose $\lambda_i < \lambda_j$. Taking $s=0$ we get $-l_ja_j(\zeta_j^{i,\lambda_i-1})^\ast \in \ad^\ast(\h_e)\alpha$. Taking successively larger values of $s$ we obtain $\zeta_j^{i,\lambda_i-1-s} \in \ad^\ast(\h_e)\alpha$ for $s=0,1,...,\lambda_i-1$ by induction. Now suppose $\lambda_i > \lambda_j$. It follows that  $j' \neq i < j$ so that $i' < j'$. By part 3 of Lemma~\ref{spanningdetails} we have $\zeta_i^{j,s} = \varepsilon_{i,j,s} \zeta_{j'}^{i',s}$ which lies in $\ad^\ast(\h_e) \alpha$ by our previous remarks. We conclude that $\spn\{(\zeta_i^{j,s})^\ast : i \neq j\}\subseteq \ad^\ast(\h_e) \alpha$ and $d_{(1, \alpha)} \varphi (\h_e, \mathfrak{h}^\ast \cap \h_e^\ast) = \h_e^\ast$, which completes the proof.
\end{proof}
We now state the equivalent proposition for orthogonal algebras. The reader will notice that we denote the canonical representation of $K_e$ in $\pp_e^\ast$ by $\Ad^\ast$. Technically this representation is the restriction of the coadjoint representation of $G_e$, although this will cause no confusion.
\begin{prop}\label{alphaKegeneric2}
The following map is dominant when $\epsilon = 1$
\begin{eqnarray*}
\vartheta : & K_e \times (\mathfrak{h}^\ast\cap \pp_e^\ast) \ra \pp_e^\ast \\
\vartheta : & (g, \gamma) \mapsto \Ad^\ast(g) \gamma. 
\end{eqnarray*}
\end{prop}
\begin{proof}
Following the previous line of reasoning, the proof hinges on showing that $\spn\{(\eta_i^{j,s})^\ast : i \neq j\}\subseteq \ad^\ast(\h_e) \alpha$. Once again we shall need explicit calculations. From formula (\ref{adform}) it follows that
\begin{eqnarray*}
\ad^\ast (\zeta_i^{j,s}) (\eta_k^{l,r})^\ast & = & \delta_{ik} (\eta_j^{l, \lambda_j - 1 + r-s})^\ast - \delta_{jl} (\eta_k^{i, \lambda_i - 1 +r-s})^\ast \\
& - & \delta_{il'} \varepsilon_{klr}(\eta_j^{k', \lambda_j - 1+r-s})^\ast - \delta_{jk'} \varepsilon_{klr} (\eta_{l'}^{i,\lambda_i - 1 +r-s})^\ast 
\end{eqnarray*}
Recall that there is an expression for $\alpha$ in terms of $(\eta_i^{i,0})^\ast$ derived in Lemma~\ref{albeso}. Using this with the above expression for $\ad^\ast (\zeta_i^{j,s}) (\eta_k^{l,r})^\ast$ and going through a series of calculations almost identical to those in Proposition~\ref{alphaKegeneric} we arrive at the assertion $$\ad^\ast(\zeta_i^{j,s}) \alpha = l_i a_i (\zeta_j^{i,\lambda_j-1-s})^\ast - l_j a_j (\zeta_j^{i,\lambda_i-1-s})^\ast.$$ The proof then concludes by making precisely the same observations as those concluding the previous proposition.

\end{proof}

We may now supply a proof Theorem~\ref{genstabal}.
\begin{proof}
In the notation of Lemma~\ref{genreg} we must show that $\varphi : G_e \times \V_\alpha \ra \g_e^\ast$ is dominant where $\V_\alpha = \{\gamma \in \g_e^\ast : (\g_e)_\alpha \subseteq (\g_e)_\gamma\}$. By \cite[Theorem~1]{Yak1} we have $(\g_e)_\alpha = \mathfrak{h}$. Since $\mathfrak{h}$ is abelian and stabalises $\n^-$ and $\n^+$, any linear form $\gamma \in \mathfrak{h}^\ast$ is annihilated by $\mathfrak{h}$. As such $\mathfrak{h}^\ast \subseteq \V_\alpha$. By Proposition~\ref{shost} the map $\vartheta = \varphi|_{G_e\times \mathfrak{h}^\ast}$ is dominant, and so too is $\varphi$. By Lemma~\ref{genreg}, part 1 of the theorem follows. 

The proofs of part 2 and 3 are similar. We shall reset the definition of $\varphi$, $\vartheta$ and $\V_\alpha$. Let $\epsilon = -1$, let $\V_\alpha = \{ \gamma\in \h_e^\ast : (\h_e)_\alpha \subseteq (\h_e)_\gamma\}$ and let $\varphi : K_e \times \V_\alpha \ra \h_e^\ast$. Since $(\h_e)_\alpha = \hh \cap \h_e$ is abelian and preserves $\spn\{\zeta_i^{j,s} : i\neq j\}$ we have $\hh^\ast\cap \h_e^\ast \subseteq \V_\alpha$. Then by Proposition~\ref{alphaKegeneric} the map $\vartheta = \varphi|_{K_e \times (\hh^\ast\cap\h_e^\ast)}$ is dominant and so is $\varphi$. Then by Lemma~\ref{genreg}, part 2 of the current theorem follows.

For part 3 we set $\epsilon = 1$, $\V_\alpha = \{\gamma \in \pp_e^\ast : (\h_e)_\alpha\subseteq (\h_e)_\gamma\}$ and $\varphi : K_e \times \V_\alpha \ra \pp_e^\ast$. In this case $(\h_e)_\alpha = \hh\cap \h_e$ annihilates $\hh \cap \pp_e$ and stabalises $\spn \{ \eta_i^{j,s} : i\neq j\}$ so $\hh^\ast\cap \pp_e^\ast \subseteq \V_\alpha$ and so by Proposition~\ref{alphaKegeneric2} the map $\varphi$ is dominant. The theorem follows by Lemma~\ref{genreg}.
\end{proof}

\begin{cor}\label{orthindex}
$\ind(\h_e, \pp_e) = \frac{1}{2}(N - |\{i : \lambda_i \text{ \emph{odd}}\}|)$ when $\epsilon = 1$.
\end{cor}
\begin{proof}
By part 3 of Theorem~\ref{genstabal} we know that $\ind(\h_e, \pp_e) = \dim(\h_e)_\alpha = \dim(\mathfrak{h}\cap \h_e) = \dim \spn\{\zeta_i^{i,s}: 1\leq i \leq n, 0 \leq s < \lambda_i\}$. Now $\zeta_i^{i,s} = 0$ only if $\xi_i^{i,\lambda_i-1-s} = -\varepsilon_{i,i,s} \xi_{i'}^{i',\lambda_i-1-s}$, which is only if $i=i'$ and $\varepsilon_{i,i,s} = 1$. In this case $\lambda_i$ is odd, by Lemma~\ref{nilpotents}, and $\varepsilon_{i,i,s} = (-1)^{\lambda_i-s} = 1$ implies $s$ is odd. Hence for each $i=i'$ we obtain $\frac{\lambda_i-1}{2}$ nonzero maps $\zeta_i^{i,s}$. In case $i\neq i'$ we have the relations $\zeta_i^{i,s} = \varepsilon_{i,i,s} \zeta_{i'}^{i',s}$ by part 3 of Lemma~\ref{spanningdetails}. Since these are the only relations we have $\dim \spn \{\zeta_i^{i,s}, \zeta_{i'}^{i',s} : 0 \leq s < \lambda_i\} = \lambda_i$ for each pair $(i,i')$ with $i\neq i'$. Since $\lambda_i$ is even in this case we conclude that $$\ind(\h_e, \pp_e) = \sum_{i} \lfloor\frac{\lambda_i}{2}\rfloor = \frac{1}{2}(N - |\{i : \lambda_i \text{ odd}\}|).$$
\end{proof}

\section{Mil'ner's Map: Symmetrisation in Positive Characteristic}\label{milnersec}
Theorem~\ref{main2} regards the structure of the invariant subalgebras of the enveloping algebra for centraliser in a Lie algebra of type $A$ or $C$. Let $\epsilon = -1$ and let $x \in \g$ (resp. $x \in \h$). The properties of $U(\g_x)^{\g_x}$ (resp. of $U(\h_x)^{\h_x}$) are deduced from those of the corresponding symmetric invariant algebras by the use of a filtration preserving isomorphism of $G_x$-modules (resp. $K_x$-modules). Such an isomorphism is known to exist over any field of characteristic zero, and is named the symmetrisation map \cite[$\S 2.4.6$]{Dix}, however over fields of positive characteristic the construction of the symmetrisation map fails and we use an approach involving representation theory. Let $\tG$ be an arbitrary algebraic group over $\K$ with $\tg = \Lie(\tG)$. Under certain mild assumptions on $\tG$ we are able to define a $\tG$-module isomorphism via an explicit composition
$$U(\tg) \ra S(U(\tg)) \ra S(\tg).$$

There is a grading $S(\tg) = \oplus_{k\geq 0} S(\tg)_k$ where $S(\tg)_k$ consists of all homogeneous polynomials of degree $k$. Since $\tg$ and $\tG$ preserve this grading the invariant subalgebras are homogeneously generated and contain the homogeneous parts of their elements. Also note that $S(\tg)$ is a filtered $\tG$-module. The Poincar\'{e}-Birkhoff-Witt (PBW) theorem  implies that there exists a canonical filtration $\{U(\tg)_{(k)}\}_{k\geq 0}$ where $U(\tg)_{(0)} =\K$ and $U(\tg)_{(k)}$ is generated by expressions $v_{i_1} v_{i_2}\cdots v_{i_k}$ with each $v_{i_j}$ in $\tg$ for $k>0$. The graded algebra associated to $U(\tg)$ shall be denoted $\gr U(\tg)$ and is canonically isomorphic to $S(\tg)$ (we shall usually identify them). In general, if $M$ and $M'$ are filtered $\tG$-modules and $\beta : M \ra M'$ is a filtered $\tG$-map then there is an associated graded $\tG$-map $\gr M \ra \gr M'$ which we denote by $\gr \beta$.

We introduce an auxiliary group. Suppose that $W$ is a finite dimensional vector space and $\tG \subseteq \tR \subseteq GL(W)$. The inclusion $\tg \subseteq \tr$ induces inclusions $S(\tg) \subseteq S(\tr)$ and $U(\tg) \subseteq U(\tr)$. Let $v_1,...,v_n$ be an ordered basis for $\tr$ and $v_{j_1}v_{j_2} \cdots v_{j_m}$ ($j_1\leq j_2\leq \cdots \leq j_m$) a basis element for $U(\tr)$. If $I \subseteq \{1,...,m\}$ then $v_I := \prod_{k \in I} v_{j_k}$ is the monomial in $U(\tr)$, with terms ordered as above. Our map $\mu: U(\tr) \ra S(U(\tr))$ acts on this basis element by the rule
\begin{eqnarray*}
\mu(v_{j_1}v_{j_2} ... v_{j_m}) = \sum_{I_1\sqcup \cdots\sqcup I_k = \{1,...,m\}} v_{I_1}\circ v_{I_2}\circ \cdots \circ v_{I_k}
\end{eqnarray*}
Here $\circ$ denotes symmetric multiplication in $S(U(\tr))$, and we do not distinguish between two partitions $I_1\sqcup \cdots\sqcup I_k$ and $I'_1\sqcup \cdots\sqcup I'_k$ of $\{1,...,m\}$ if there exists $\tau$ in $\mathfrak{S}_k$, the symmetric group on $k$ letters, such that $I_j = I'_{\sigma(j)}$ for $1\leq j\leq k$.  The map is extended by linearity to all of $U(\tr)$ and enjoys several properties, most notably: $\mu$ is an injective map of $\tR$-modules fulfilling
\begin{eqnarray*}
\mu(v_{j_1}v_{j_2} ... v_{j_m}) = \mu(v_{j_1})\mu(v_{j_m})...\mu(v_{j_m})\mod S(U(\tr))_{(m-1)}.
\end{eqnarray*}
This map was constructed by Mil'ner in \cite{Mil}. In that paper he attempted to settle the first Kac-Weisfeiler conjecture in the affirmative, although the argument was eventually found to contain a gap which could not be closed. Nonetheless, much good theory has come out of that attempt and the results of \cite{FP} were accrued in the process of trying to make sense of that article. The main point of reference for our purposes is \cite[$\S 3$]{Pre} where the map is considered especially in the case of centralisers.

The next ingredient is a map $S(U(\tr)) \ra S(\tr)$. This shall be induced from a map $U(\tr) \ra \tr$ as follows. We say that $\tr$ \emph{possesses Richardson's property} if there exists an $\Ad(\tR)$-invariant decomposition $\gl(W) = \tr \oplus \mathfrak{c}$. The inclusion $\tr \subseteq \gl(W)$ implies the existence of an injective map $U(\tr) \ra U(\gl(W))$. By the universal property of enveloping algebras, there is a surjection $U(\gl(W))\ra \gl(W)$ induced by the module $W$. The composition $U(\tr) \ra U(\gl(W)) \ra \gl(W)$ acts as the identity on $\tr \subseteq U(\tr)$, and so the image of this map contains $\tr$. Composing with the projection onto the first factor in the decomposition $\gl(W) = \tr \oplus U$ we obtain an $\Ad(\tR)$-equivariant map $$\pi : U(\tr) \ra \tr.$$ The symmetric algebra construction is a covariant functor from $\K$-vector spaces to commutative $\K$-algebras, and so we obtain a map $$S(\pi) : S(U(\tr)) \ra S(\tr).$$
Continuing to assume that $\tr$ possesses Richardson's property, define $\beta : U(\tr) \ra S(\tr)$ as the composition $\beta = S(\pi) \circ \mu$. We say that the subalgebra $\tg$ is \emph{saturated} provided the image of $\pi|_{U(\tg)}$ is contained in $\tg$ (note that the reverse inclusion holds automatically). Now the following theorem is contained in Property (B1), Proposition~3.4 and Lemma~3.5 of \cite{Pre}.
\begin{thm}\label{milmap}
If $\tG \subseteq \tR \subseteq GL(W)$, $\tr$ possesses Richardson's property and $\tg$ is a saturated subalgebra, then $\beta : U(\tg) \ra S(\tg)$ is an isomorphism of $\tG$-modules such that the associated graded map $$\emph\gr (\beta) :  \emph\gr U(\tg) \cong S(\tg) \ra \emph\gr S(\tg) = S(\tg)$$ is the identity. Furthermore, if $\tg$ is reductive and linear then it possesses Richardson's property, and if $x\in \tg$ then $\tg_x$ is a saturated subalgebra of $\tg$. 
\end{thm}
\begin{rem}
\rm{By insisting in Richardson's property that the decomposition is $\emph{\Ad}(\tR)$-invariant we obtain a $\tG$-module isomorphism in the above theorem. In Premet's proof the decomposition is only assumed to be $\ad(\tr)$-invariant, implying that $\beta$ is a map of $\tg$-modules. Inspecting the details of \cite[$\S 3$]{Pre} we can see that this this slight alteration will place no extra burden upon the proof.}
\end{rem}

\section{Proof of the Main Theorems}\label{Proofs}
\subsection{Symmetric Invariants}

The deductions of Sections~\ref{elemInv} through to \ref{GenEl} aim to satisfy the assumptions of a theorem of Serge Skryabin. We shall record here the first part of that theorem, which is sufficient for our purposes.
\begin{thm}\label{skryabin}
\cite[Theorem~5.4(i)]{Skr}  \\
Suppose that $X$ is a smooth affine variety and $\tg$ is a $p$-Lie algebra acting upon $\K[X]$ by derivations. Let $f_1, ..., f_m \in \K[X]^{\tg}$ where $$m = m(\tg, X) := \dim(X) - \dim(\tg) + \emph{\ind}(\tg, X).$$ Denote by $J$ the closed set of those $x \in X$ such that the differentials $d_x f_1,..., d_x f_m$ are linearly dependent. If $\text{\emph{codim}}_X  J \geq 2$ then $$\K[X]^{\tg} = \K[X]^{p}[f_1,...,f_m]$$ and $\K[X]$ is free of rank $p^m$ over $\K[X]^{p}$.
\end{thm}

The algebra $\K[X]^{p} = \{ f^{p} : f \in \K[X]\}$ is called the $p^{\text{th}}$ power subalgebra of $\K[X]$. Before we can apply this theorem we shall need a lemma.
\begin{lem}\label{noinv} The following are true:
\begin{enumerate} 
\item{$m(\g_e, \g_e^\ast) = N$;}
\smallskip
\item{$m(\h_e, \h_e^\ast) = |\{x_r : r \text{ even}\}| = N/2$ when $\epsilon = -1$;}
\smallskip
\item{$m(\h_e, \pp_e^\ast) = |\{x_r : r + d_r \text{ even}\}| = \frac{1}{2}(N + |\{i : \lambda_i \text{ odd}\}|)$ when $\epsilon = 1$.}
\end{enumerate}
Here $m$ is defined in the statement of Skryabin's theorem.
\end{lem}
\begin{proof}
We have $m(\g_e, \g_e^\ast) = \ind(\g_e, \g_e^\ast) = N$ by \cite{Yak1}. If we take $\epsilon = -1$ so that $N$ is even then $m(\h_e, \h_e^\ast) = \ind(\h_e, \h_e^\ast) = N/2 = |\{ 1\leq r \leq N : r \text{ even}\}|$ by the same reasoning.

Now let $\epsilon = 1$. By \cite[$\S 3.2,~(3)$]{Jan} we have $$\dim(\h_e) = \frac{1}{2}( \dim(\g_e) - |\{ i : \lambda_i \text{ odd}\}|).$$ Since $\dim\pp_e^\ast = \dim \pp_e = \dim\g_e - \dim\h_e$ we have $\dim \pp_e^\ast - \dim \h_e = |\{ i : \lambda_i \text{ odd}\}|.$ By Corollary \ref{orthindex} we get $m(\h_e, \pp_e^\ast) = \frac{1}{2}(N + |\{i : \lambda_i \text{ odd}\}|)$. We must show that this number equals $|\{1\leq r \leq N: r + d_r \text{ even}\}|$. Partition the set $\{1,...,N\}$ into disjoint subsets $\II_i = \{1\leq r \leq N : d_r = i\}$ where $i=1,...,n$. Then $|\II_i| = \lambda_i$ and by the definition of the sequence $d_1,d_2,...,d_N$ we have $r \in \II_i$ if and only if
\begin{eqnarray}\label{IIprop}
0 < r - \sum_{k=1}^{i-1} \lambda_k \leq \lambda_i.
\end{eqnarray}
If $i \neq i'$ then $|\II_i| = \lambda_i$ is even by Lemma \ref{nilpotents}. In this case, regardless of the parity of $\lambda_i$ there are exactly $\lambda_i/2$ values $r$ fulfilling (\ref{IIprop}) with $r+d_r$ even. Now suppose $i = i'$ so that $|\II_i| = \lambda_i$ is odd. We consider two cases: $i$ odd or $i$ even. In the first case there must be an even number of indexes $j$ with $1 \leq j < i$ and $j = j'$, since $j' \in \{j-1, j, j+1\}$. Thus there an even number of indexes $1 \leq j < i$ with $\lambda_j$ odd by Lemma \ref{nilpotents}. We deduce that $\sum_{k=1}^{i-1} \lambda_k$ is even. If $r \in \II_i$ so that $d_r = i$, then $r + d_r$ even implies $r$ is odd. Thus there are exactly $(\lambda_i + 1)/2$ values of $r$ fulfilling (\ref{IIprop}) with $r+d_r$ even. Now suppose $i=i'$ and $i$ is even. Similar to the case $i$ odd, there must be an odd number of indexes $1\leq j < i$ with $\lambda_j$ odd. Thus $\sum_{k=1}^{i-1}\lambda_k$ is odd. For $r \in\II_i$, If $r+d_r$ is even then $r$ is even. Thus there are exactly $(\lambda_i + 1)/2$ values of $r$ fulfilling (\ref{IIprop}) with $r+d_r$ even. With Lemma \ref{nilpotents} in mind we are able to conclude that $$|\{1\leq r \leq N : r + d_r \text{ even}\}| = \sum_{i\neq i'} \lambda_i/2 + \sum_{i=i'} (\lambda_i+1)/2 = \frac{1}{2}(N + |\{i : \lambda_i \text{ odd}\}|).$$
\end{proof}

The proofs are identical for both parts of Theorem~\ref{main1} and also for Theorem~\ref{main3}, and so we may unify notation. In the remnant of this subsection, $Q$ shall be our underlying group, $e \in \tg = \Lie(\tG)$ shall be a choice of nilpotent element, $\epsilon = \pm 1$ will indicate whether $K$ is orthogonal or symplectic (this, of course, is redundant when $Q$ is of type $A$), $X$ shall be the $Q_e$-module of interest, $\mathcal{I}$ shall be a finite subset of $\K[X]^{\tG_e}$ and $m$ shall be the integer defined in the statement of Skryabin's theorem. The following table explains how our notations are unified:
\begin{center}
  \begin{tabular}{ | c | c | c | c | c | c|}
    \hline
    & $Q$ & $\epsilon$ & $X$ & $\mathcal{I}$ & $m$\\ \hline\hline
    \text{Case 1} & $G$ & $\pm 1$ & $\g_e^\ast$ & $\{x_1,...,x_N\}$ & $m(\g_e, \g_e)$  \\ \hline
    \text{Case 2} & $K$ & $-1$ &$ \h_e^\ast$ & $\{x_r|_{\h_e^\ast} : r \text{ even}\}$& $m(\h_e, \h_e)$ \\ \hline
    \text{Case 3} & $K$ & $1$ & $\pp_e^\ast$ & $\{x_r|_{\pp_e^\ast} : r + d_r \text{ even}\}$ & $m(\h_e, \pp_e)$\\
    \hline
  \end{tabular}
\end{center}
Theorems~\ref{main1} and \ref{main3} are both convenient ways of stating our first results without discussing the explicit constructions of the invariant algebras. Our method tells us more than is expressed by those theorems, especially in the nilpotent case. Using the above notations we may precisely state our result as follows.
\begin{thm}\label{main13}
Suppose that we are in Case 1, 2 or 3. If $e\in \tg$ then $|\mathcal{I}| = m$. Furthermore $\K[X]^{Q_e} = \K[\mathcal{I}]$ is a polynomial algebra generated by $\mathcal{I}$, and $$\K[X]^{\tg_e} = \K[X]^p[\mathcal{I}]$$ is a free $\K[X]^p$-module of rank $p^m$.
\end{thm}

\begin{proof}
That $|\mathcal{I}| = m$ follows immediately from Lemma~\ref{noinv} and Corollary~\ref{nonzero1}.

We now prove $\K[X]^{\tg_e} = \K[X]^p[\mathcal{I}]$ is a free $\K[X]^p$-module of rank $p^m$ by applying Skryabin's theorem (Theorem~\ref{skryabin}). Corollary~\ref{invariants} tells us that the polynomials $\mathcal{I}$ are invariant. Lemma~\ref{noinv} ensures that $\mathcal{I}$ is of the correct size. Propositions \ref{Codimension 2} and \ref{jaclocsymorth} confirm that the condition on codimension of the Jacobian locus is satisfied, and so the the assumptions of Skryabin's theorem are satisfied. The claim follows.

In order to prove that $\K[X]^{Q_e} = \K[\mathcal{I}]$ we use induction on polynomial degree. Let $f \in \K[X]^{\tG_e}$. If $\deg(f) < p$ then by the previous paragraph, $f \in \K[\mathcal{I}]$. Suppose $\deg(f) \geq p$ and that all $g \in \K[X]^{\tG_e}$ with $\deg(g) < \deg(f)$ actually lie in $\K[\mathcal{I}]$. We shall call a product of elements of $\X$ a monomial in $\X$, and this terminology shall cause no confusion. Since $f \in \K[X]^{\tG_e} \subseteq \K[X]^{\tg_e}$ we may write $f$ as a sum of monomials in $\X$ with coefficients in $\K[X]^p$. Since $f$ and all monomials in $\mathcal{I}$ are fixed by $\tG_e$ we deduce that each coefficient is also fixed by $\tG_e$. Since $(\K[X]^p)^{\tG_e} = (\K[X]^{\tG_e})^p$ we conclude by induction that each coefficient is a $p^\text{th}$ power of an element of $\K[\mathcal{I}]$. It follows that $f \in \K[\mathcal{I}]$.
\end{proof}
\begin{rem}
\rm{Case 3 of Theorem~\ref{main13} quickly implies Theorem~\ref{main3}. Cases 1 and 2 imply the nilpotent case of Theorem~\ref{main1}. As was noted in the discussion following the statement of that theorem, there is a simple reduction to the nilpotent case. Therefore Theorems~\ref{main1} and \ref{main3} follow immediately from Theorem~\ref{main13}.}
\end{rem}

\subsection{The Centre of the Enveloping Algebra}
In this section we shall prove Theorem~\ref{main2}, using the theory introduced in Section~\ref{milnersec} and a simple filtration argument. Fortunately the proof in type $A$ and $C$ is identical and may be dealt with at once. Unfortunately we will have to reset our notation again. We let $\epsilon = -1$, let $\tG \in \{G, K\}$ and choose $x \in \tg$. Thanks to Theorem~\ref{milmap} there is a filtered $\tG_x$-module isomorphism $$\beta : U(\tg_x) \ra S(\tg_x).$$ According to Theorem~\ref{main1}, $S(\tg_x)^{Q_x}$ is polynomial on $\rank(\tg)$ generators. In keeping with the previous section we denote a set of homogeneous generators by $\mathcal{I}$. Since $\gr (\beta)$ is the identity, the top graded component of $\beta^{-1} f$ is equal to $f$ for all $f \in \mathcal{I}$. We shall denote by $\widetilde{\mathcal{I}}$ the preimage under $\beta$ of $\mathcal{I}$ in $U(\tg_x)$.  Of course $\tcx \subseteq U(\tg_x)^{\tG_x} \subseteq Z(\tg_x)$. We are now ready to present a proof of Theorem~\ref{main2}
\begin{proof}
The graded algebra of the centre $\gr Z(\tg_x)$ is contained in $S(\tg_x)^{\tg_x}$, which is equal to $S(\tg_x)^p[\mathcal{I}]$ by the previous section. Since $\mathcal{I}$ consists of homogeneous polynomials we have $S(\tg_x)^p[\mathcal{I}] = \gr Z_p(\tg_x)[\tcx]$ and furthermore $Z_p(\tg_x)[\tcx]$ is central in $U(\tg_x)$. Placing all of these inclusions together we get
$$\gr Z(\tg_x) \subseteq S(\tg_x)^{\tg_x} = S(\tg_x)^p[\mathcal{I}] \subseteq \gr Z(\tg_x)[\tcx] \subseteq \gr Z(\tg_x).$$
Thence $\gr Z(\tg_x)[\tcx] = \gr Z(\tg_x)$ and the dimensions of the graded components of $\gr Z(\tg_x)[\tcx]$ and $\gr Z(\tg_x)$ coincide. Since $Z(\tg_x)[\tcx] \subseteq Z(\tg_x)$ we deduce that this inclusion is an equality. Part 1 of the theorem follows.

Our proof of part 2 is very similar. We denote by $\K[\tcx]$ the subalgebra of $Z(\tg_x)$ generated by $\tcx$, and similar for $\K[\mathcal{I}]$. We have $$\K[\mathcal{I}] = \gr \K[\tcx] \subseteq \gr U(\tg_x)^{\tG_x} \subseteq S(\tg_x)^{\tG_x} = \K[\mathcal{I}]$$ and so we have equality throughout. In particular the dimensions of the graded components of $\gr U(\tg_x)^{\tG_x}$ and $\gr \K[\tcx]$ coincide. Since $U(\tg_x)^{\tG_x}\subseteq \K[\tcx]$ we must actually have equality $U(\tg_x)^{\tG_x} = \K[\tcx]$ and part 2 follows.

Part 3 is an immediate consequence of parts 1 and 2, and the fact that the $Z_p(\tg_x)$-basis of $Z(\tg_x)$ is also a $Z_p(\tg_x)^{\tG_x}$-basis of $Z(\tg_x)$.
\end{proof}

\subsection{Dimensions of Simple Modules}
We remind the reader that the dimensions of simple modules for modular Lie algebras are bounded. This fact essentially follows from the fact that the enveloping algebra is a finite module over its centre, and sits in stark contrast to the characteristic zero case, where each simple algebra has simple modules of arbitrarily high dimension. For each Lie algebra $\tg$ over $\K$ we denote by $M(\tg)$ the maximal dimension of simple modules. The first Kac-Weisfeiler conjecture asserts that $$M(\tg) = p^{\frac{1}{2}(\dim(\tg) - \ind(\tg))}.$$ This conjecture is extremely general and should be approached with some trepidation. See the introduction for a further discussion.

As was explained in the introduction, Zassenhaus interpreted the number $M(\tg)$ in algebraic terms. Let $F_p(\tg)$ be the full field of fractions of $Z_p(\tg)$ and let $F(\tg)$ be that of $Z(\tg)$. We record the following lemma which condenses some of the content of \cite[Theorems~1 \& 6]{Zas}, and is explained concisely in \cite[A.5]{Jan.}.
\begin{lem}\label{Zassen}
The first Kac-Weisfeiler conjecture is equivalent to the assertion $$[F(\tg) : F_p(\tg)] = p^{\ind(\tg)}.$$
\end{lem}
Using this lemma, and our results on invariant theory, we now prove Theorem~\ref{KW1} which states that the first Kac-Weisfeiler conjecture holds for centralisers in types $A$ and $C$. We actually give two proofs, very different in nature. The first of these is more general and applies to all centralisers in types $A$ and $C$ whilst the second proof only applies only in the case $x$ nilpotent with $\LL$ of type $A$. We now supply the first proof, which is purely algebraic and makes use of Lemma~\ref{Zassen}.
\begin{proof}
Let $\tG \in \{G, K\}$ and $x \in \tg$. Define $F := F(\tg_x)$ and $F_p := F_p(\tg_x)$. By the previous lemma we may prove that $[F : F_p] = p^{\ind(\tg)}$. In the notation of the previous subsection we let $F' = F_p [\tcx] = Z(\LL_x)\otimes_{Z_p(\LL_x)} F_p$. Clearly $F' \subseteq F$. We claim that this inclusion is actually an equality. It will suffice to prove that every element of $F$ can be written in the form $f/g$ where $f \in Z(\LL_x)$ and $g \in Z_p(\LL_x)$. This is actually a very general result which holds whenever we have a module-finite inclusion of $\K$-algebras. We shall recall the proof for the reader's convenience.

Suppose $f/g \in F$. Since $Z(\tg_x)$ is a finite $Z_p(\tg_x)$-module we have $a_n g^n  + a_{n-1} g^{n-1} + \cdots + a_0 = 0$ for some $n \in \N$ and certain coefficients $a_i \in Z_p(\LL_x)$ with $a_0 \neq 0$. We conclude that $$\frac{f}{g} = \frac{f(a_ng^{n-1} + a_{n-1}g^{n-2} + \cdots + a_1)}{g(a_ng^{n-1} + a_{n-1}g^{n-2} + \cdots + a_1)} =  \frac{f(a_ng^{n-1} + a_{n-1}g^{n-2} + \cdots + a_1)}{-a_0} \in F'.$$
Therefore $F = F'$ as claimed. Since $Z(\tg_x)$ is free of rank $p^{\ind(\tg)}$ over $Z_p(\tg_x)$ (Theorem~\ref{main2}) we conclude that $F$ is of dimension $p^{\ind(\tg)}$ over $F_p$. Thanks to \cite{Yak1} this is equal to $p^{\ind(\tg_x)}$ and the theorem follows.
\end{proof}

Before we record a second proof of the corollary we must first recall the rudiments of reduced enveloping algebras. If $W$ is a simple $\LL_x$-module then Schur's lemma asserts that $Z(\LL_x)$ acts upon $W$ by scalars. In particular, $Z_p(\LL_x)$ acts by scalars. The semilinearity of the $p$-operation $v \mapsto v^{[p]}$ implies that for each such module $W$ there exists $\chi \in \LL_x^\ast$ fulfilling $(v^p - v^{[p]}) w = \chi(v)^p w$ for all $v \in \LL_x, w \in W$. We call $\chi$ the $p$-character of $W$ and define the reduced enveloping algebra $$U_\chi(\LL_x) = U(\LL_x)/ I_{\chi}$$ where $I_\chi$ denotes the (left and right) ideal of $U(\LL_x)$ generated by expressions of the form $v^p - v^{[p]} - \chi(v)^p$ with $v\in \LL_x$. Every simple $\LL_x$-module with $p$-character $\chi$ is canonically a simple $U_\chi(\LL_x)$-module and vice versa. Since the following proof of the first Kac-Weisfeiler conjecture applies only to the centraliser of a nilpotent element in type $A$ we shall revert to the notation of the introduction, with $e \in \g$ nilpotent. Recall the decompositions $\g_e = \n^- \oplus \hh \oplus \n^+$ and $\g_e^\ast = (\n^-)^\ast \oplus \hh^\ast \oplus (\n^+)^\ast$ from Section~\ref{JacLoc1}.
\begin{proof}
Define
\begin{eqnarray*}
\Omega = \{ \chi \in \g_e^\ast : \dim(W) = M(\g_e) \text { for all simple } U_\chi(\g_e) \text{-modules } W \}
\end{eqnarray*}
Let $\chi \in \g_e^\ast$. Every $g \in G_e$ acts on $\g_e^\ast$ and induces an isomorphism of algebras
\begin{eqnarray*}
U_\chi(\g_e) \xrightarrow{\sim} U_{\tau(\chi)}(\g_e)
\end{eqnarray*}
Hence $\Omega$ is $G_e$-stable. In \cite[Proposition~4.2(1)]{PS} it was shown that $\Omega$ is a non-empty Zariski open subset of $\mathfrak{g}_e^\ast$. By Theorem~\ref{genstabal} there exists a nonempty open subset $\mathcal{O} \subseteq\g_e^\ast$ such that the $\g_e$-stabalisers of points in $\mathcal{O}$ are $G_e$-conjugate to $\mathfrak{h} = \{\xi_i^{i,\lambda_i - 1-s} :1\leq i\leq n, 0 \leq s < \lambda_i\}$. Since $\mathfrak{h}$ is abelian and stabalises $\mathfrak{n}^-\oplus \mathfrak{n}^+$ the stabalisers of points in $\mathfrak{h}^\ast$ include $\mathfrak{h}$.

Hence we can find $\chi \in \Omega \cap \mathfrak{h}^\ast$, ie. a linear function vanishing on $(\mathfrak{n}^-)^\ast$ and $(\mathfrak{n}^+)^\ast$ with the additional property that every irreducible $U_\chi(\mathfrak{g}_e)$-module has dimension $M(\mathfrak{g}_e)$. Identify $\chi$ with its restriction to $\mathfrak{h}^\ast$. The decomposition of $\g_e$ induces a decomposition of enveloping algebras
\begin{eqnarray*}
U_\chi(\mathfrak{g}_e) = U_0(\mathfrak{n}^-) \otimes_\K U_\chi(\mathfrak{h}) \otimes_\K U_0(\mathfrak{n}^+)
\end{eqnarray*}
Let $W$ be a simple $U_\chi(\mathfrak{g}_e)$-module so that $\dim(W) = M(\g_e)$. By Engel's theorem $\mathfrak{n}^+$ has a common eigenvector in $W$ of eigenvalue 0 and $\mathfrak{h}$ acts diagonally on $W^{\mathfrak{n}^+} \neq 0$. We can find a 1-dimensional $U_\chi(\mathfrak{h}) \otimes_\K U_0(\mathfrak{n}^+)$-module $\K w_\mu \subseteq W^{\mathfrak{n}^+}$ where $h w_\mu = \mu(h) w_\mu$ for all $h \in \mathfrak{h}$. 

The induced module $\Ind_{\mathfrak{h} + \mathfrak{n}^+}^{\mathfrak{g}_e} (\K w_\mu) = U_\chi(\mathfrak{g}_e) \otimes_{U_\chi(\mathfrak{h}) \otimes U_0(\mathfrak{n}^+)} \K w_\mu$ has dimension $\dim(U_0(\mathfrak{n}^-)) = p^{\dim(\mathfrak{n}^-)}$. We have $\dim(\mathfrak{n}^-) = \dim(\mathfrak{n}^+)$ and $\dim(\mathfrak{h}) = N = \ind(\g_e)$ by \cite[Theorem~1]{Yak1}. We obtain $\dim(\g_e) = 2\dim(\mathfrak{n}^-) + \ind(\g_e)$. Therefore
$$\dim(\Ind_{\mathfrak{h} + \mathfrak{n}^+}^{\mathfrak{g}_e} (\K w_\mu)) = p^{\frac{1}{2}(\dim(\mathfrak{g}_e) - \ind(\mathfrak{g}_e))}.$$
By standard theory of induced modules, $W$ is a homomorphic image of $\Ind_{\mathfrak{h} + \mathfrak{n}^+}^{\mathfrak{g}_e} (\K w_\mu)$ so the dimension of the former is bounded above by that of the latter. According to \cite[Theorem~5.4(2)]{PS} it is possible to choose $\chi$ so that all finite dimensional $U_\chi(\g_e)$-modules have dimension divisible by $p^{\frac{1}{2}(\dim(\mathfrak{g}_e) - \ind(\mathfrak{g}_e))}$. This ensures that $W = \Ind_{\mathfrak{h} + \mathfrak{n}^+}^{\mathfrak{g}_e} (\K w_\mu)$, and that $M(\g_e) = p^{\frac{1}{2}(\dim(\mathfrak{g}_e) - \ind(\mathfrak{g}_e))}$. 
\end{proof}
\begin{rem}
\rm{Although our second proof is much less general it does have the virtue of implying that the Verma modules for $\g_e$ are generically simple. Part 3 of Lemma~\ref{spanningdetails} ensures that there is no `triangular' decomposition for classical types other than type $A$, and it seems this proof will extend no further.}
\end{rem}

\subsection{The Zassenhaus Variety}

Retain the notations of theorem \ref{main2}. In this final section we bring together some related results to offer an algebraic characterisation of the singular points on the Zassenhaus variety $\mathcal{Z} = \Specm (Z(\LL_x))$ when $Q$ is of type $A$ or $C$. This scheme of argument was first sketched in \cite[Remark~5.2]{PS} and is quite general.
\begin{thm}\label{zassen}
If $\LL$ is a simple Lie algebra of type $A$ or $C$ over $\K$ and $x \in \LL$ then $\m$ is a smooth point of $\ZZ$ if and only if $U(\LL_x)/ \m U(\LL_x) \cong \text{\emph{Mat}}_{M(\LL_x)}(\K)$.
\end{thm}
\begin{proof}
Recall that the singular locus $(\LL_x^\ast)_\text{sing}$ is defined to be the set of all $\chi \in \LL_x^\ast$ such that $\dim((\LL_x)_{\chi}) > \ind(\LL_x)$. By \cite[Theorems~3.4 \& 3.11]{PPY}, $\codim_{\LL_x^\ast} (\LL_x^\ast)_{\text{sing}} \geq 2$ provided $x$ is nilpotent. Strictly speaking they worked over $\C$ there but the assumption was not necessary - indeed \cite{Yak1} works in good characteristic and that article is the basis for the aforementioned two theorems. Now the fact that $\codim_{\LL_x^\ast} (\LL_x^\ast)_{\text{sing}} \geq 2$ when $x$ is not nilpotent may be proven using a reduction to the nilpotent case very similar to the one used for Theorem~\ref{main1}. For a non-negative integer $m$ we let $\mathscr{X}_m$ denote the set of all $\chi \in \LL_x^\ast$ such that $U_\chi(\LL_x)$ has a module of finite dimension not divisible by $p^m$. According to \cite[Proposition~5.2,]{PS}  $\dim(\mathscr{X}_{M(\LL_x)}) \leq \dim(\LL^\ast_\text{sing})$, from whence we deduce that $\codim_{\LL_x^\ast} (\mathscr{X}_{M(\LL_x)}) \geq 2$. 

By \cite[Theorems~5 \& 6]{Zas}, we may infer that there is a closed subset $\CC \subseteq \ZZ$ such that $$\ZZ \backslash \CC = \{ \m \in \ZZ : U(\LL_x)/ \m U(\LL_x) \cong \text{Mat}_{M(\LL_x)}(\K)\}.$$ We claim that $\codim_{\ZZ} \CC \geq 2$. The inclusion $Z_p(\LL_x) \ra Z(\LL_x)$ induces as finite morphism of maximal spectra $\ZZ \ra \Specm (Z_p(\LL_x))$ and by identifying $\Specm (Z_p(\LL_x))$ and $\LL_x^\ast$ we obtain a map $\tau : \Specm Z(\LL_x ) \ra \LL_x^\ast$. Explicitly, $\m \ra \chi$ where $\chi$ is the linear functional such that $\m \cap Z_p(\LL_x) = \m_\chi$ and $U_\chi(\LL_x) = U(\LL_x) / \m_\chi U(\LL_x)$. This map is closed and has finite fibres so $\codim_{\ZZ} \CC = \codim_{\LL_x^\ast}\tau(\CC)$. If $\chi \in \tau(\CC)$ then $U_\chi(\LL_x)$ has a module of dimension less than $M(\LL_x)$ so that $\chi \in \X_{M(\LL_x)}$. We conclude that $$\codim_{\ZZ} \CC = \codim_{\LL_x^\ast}\tau(\CC) \geq \codim_{\LL_x^\ast} \X_{M(\LL_x)} \geq 2.$$ Denote by $Z_\m$ and $U_\m$ the localisation of $Z(\LL_x)$ and $U(\LL_x)$ respectively at $\m \in \ZZ$ and suppose $\m \in \CC$. Certainly $U_\m$ is a finite module over $Z_\m$. Furthermore the unique maximal ideal of $Z_\m$ is $\m Z_\m$ and $U_\m/\m U_\m \cong U(\LL_x) /\m U(\LL_x)$ is a matrix algebra over its centre $Z / \m Z \cong \K$. By \cite{DI} theorem 7.1, $U_\m$ is a separable algebra over its centre $Z_\m$ (ie. $U_\m$ is an Azumaya algebra) for all $\m \in \CC$. Combining these deductions with \cite[2.2, 2.3]{BG}, we have satisfied the assumptions of \cite[Theorem~3.8]{BG} and the result follows.
\end{proof}

\end{document}